\documentclass[a4paper,10pt,american]{amsart}

\usepackage{dsfont}
\usepackage{stmaryrd}
\usepackage{esint}
\usepackage{amssymb}
\usepackage[square,sort,comma,numbers]{natbib}

\makeatletter
\newsavebox{\@brx}
\newcommand{\llangle}[1][]{\savebox{\@brx}{\(\m@th{#1\langle}\)}%
  \mathopen{\copy\@brx\kern-0.5\wd\@brx\usebox{\@brx}}}
\newcommand{\rrangle}[1][]{\savebox{\@brx}{\(\m@th{#1\rangle}\)}%
  \mathclose{\copy\@brx\kern-0.5\wd\@brx\usebox{\@brx}}}
\makeatother

\usepackage{amsmath}
\usepackage[latin1]{inputenc}
\usepackage[american]{babel}
\usepackage{xcolor}
\usepackage{natbib}
\usepackage{verbatim}

\usepackage[colorlinks,pdfpagelabels,pdfstartview = FitH,bookmarksopen
= true,bookmarksnumbered = true,linkcolor = blue,plainpages =
false,hypertexnames = false,citecolor = red,pagebackref=false]{hyperref}

\allowdisplaybreaks

\newtheorem{theorem}{Theorem}
\newtheorem{lemma}[theorem]{Lemma}

\theoremstyle{definition}
\newtheorem{definition}[theorem]{Definition}

\theoremstyle{remark}
\newtheorem{remark}[theorem]{Remark}

\numberwithin{theorem}{section}
\numberwithin{equation}{section}

\newcommand{\mollifytime}[2]{[\![ #1 ]\!]_{#2}}
\newcommand{\N}{\mathbb{N}}
\newcommand{\R}{\mathbb{R}}

\renewcommand{\d}{\mathrm{d}}
                  
\newcommand{\dx}{\mathrm{d}x}

\newcommand{\dt}{\mathrm{d}t}
\newcommand{\ds}{\mathrm{d}s}

\newcommand{\eps}{\varepsilon}
\renewcommand{\epsilon}{\varepsilon}
\renewcommand{\rho}{\varrho}
\renewcommand{\u}{\boldsymbol{u}}

\DeclareMathAlphabet{\mathpzc}{OT1}{pzc}{m}{it}

\DeclareMathOperator{\spt}{spt}

\DeclareMathOperator*{\osc}{osc}
\DeclareMathOperator{\Div}{div}
\DeclareMathOperator*{\esssup}{ess\,sup}
\DeclareMathOperator*{\essinf}{ess\,inf}
\DeclareMathOperator*{\essosc}{ess\,osc}
\DeclareMathOperator{\loc}{loc}

\makeatletter
\@namedef{subjclassname@2020}{%
  \textup{2020} Mathematics Subject Classification}
\makeatother

\def\Xint#1{\mathchoice
    {\XXint\displaystyle\textstyle{#1}}%
    {\XXint\textstyle\scriptstyle{#1}}%
    {\XXint\scriptstyle\scriptscriptstyle{#1}}%
    {\XXint\scriptscriptstyle\scriptscriptstyle{#1}}%
    \!\int}
\def\XXint#1#2#3{\setbox0=\hbox{$#1{#2#3}{\int}$}
    \vcenter{\hbox{$#2#3$}}\kern-0.5\wd0}

\def\dashint{\Xint{\raise4pt\hbox to7pt{\hrulefill}}}

\def\XXiint#1#2#3{\setbox0=\hbox{$#1{#2#3}{\iint}$}
    \vcenter{\hbox{$#2#3$}}\kern-0.5\wd0}


\begin{document}
\title[H\"older regularity for obstacle problems]{On the H\"older regularity for obstacle problems to porous medium type equations}

\author[K. Moring]{Kristian Moring}
\address{Kristian Moring\\
Department of Mathematics and Systems Analysis, Aalto University\\
P.~O.~Box 11100, FI-00076 Aalto, Finland}
\email{kristian.moring@aalto.fi}

\author[L. Sch\"atzler]{Leah Sch\"atzler}
\address{Leah Sch\"atzler \\ 
Fachbereich Mathematik, Universit\"at Salzburg \\
Hellbrunner Str. 34, 5020 Salzburg, Austria}
\email{leahanna.schaetzler@plus.ac.at}

\begin{abstract}

We show that signed weak solutions to parabolic obstacle problems with porous medium type structure are locally H\"older continuous, provided that the obstacle is H\"older continuous.

\end{abstract}


\maketitle

\section{Introduction}
Let $\Omega_T := \Omega \times (0,T)$, where $\Omega \subset \R^n$ is an open set
and $0<T<\infty$.
In the present paper we are concerned with the obstacle problem to partial differential equations, whose prototype is the porous medium equation (PME for short)
$$
	\partial_t \big( |u|^{q-1} u \big) - \Delta u = 0
	\quad \text{in } \Omega_T
$$
with a parameter $q \in (0, \infty)$.
If $0<q<1$, the equation is degenerate and if $q>1$, it is singular.
More generally, for $q \in (0, \infty)$ we are concerned with
partial differential equations of the type
\begin{equation}
	\partial_t \big( |u|^{q-1} u \big) - \Div \mathbf{A}(x,t,u, \nabla u) = 0
	\quad \text{in } \Omega_T,
	\label{eq:pde}
\end{equation}
where $\mathbf{A} \colon \Omega_T \times \R \times \R^n \to \R^n$
is a Carath\'eodory function, i.e.~it is
measurable with respect to $(x,t) \in \Omega_T$ for all $(u, \xi) \in \R \times \R^n$
and continuous with respect to $(u, \xi) \in \R \times \R^n$ for a.e.~$(x,t) \in \Omega_T$.
Moreover, we assume that $\mathbf{A}$ satisfies the structure conditions
\begin{equation}
	\left\{
	\begin{array}{l}
		\mathbf{A}(x,t,u,\xi) \cdot \xi \geq C_o |\xi|^2, \\[5pt]
		|\mathbf{A}(x,t,u,\xi)| \leq C_1 |\xi|,
	\end{array}
	\right.
	\quad \text{for a.e. $(x,t) \in \Omega_T$ and all $(u, \xi) \in \R \times \R^n$,}
	\label{eq:structure}
\end{equation}
where $C_o, C_1 >0$ are given constants.
For the basic theory for the porous medium equation and its generalizations,
we refer to the monographs~\cite{DK,DGV,Vazquez_FDE,Vazquez_PME,WZY}.

We use a variational approach to define solutions to the obstacle problem to \eqref{eq:pde}
with an obstacle function $\psi \in C^0(\overline{\Omega_T})$. Heuristically, for a solution $u$ that is above the given obstacle $\psi$, the variational inequality
\begin{equation*}
	\iint_{\Omega_T} \left[ \partial_t ( |u|^{q-1} u )(v - u)
	+  \mathbf{A}(x,t,u,\nabla u)
	\cdot \nabla (v - u) \right] \, \d x \d t \geq 0,
\end{equation*}
must hold true for all comparison maps $v$ satisfying $v \geq \psi$. Since we do not assume regularity properties for $u$ in the time direction, the first term is defined more rigorously in the next section. Existence results for variational solutions to the obstacle problem to porous medium type equations can be found in~\cite{AL,BLS,Schaetzler1,Schaetzler2}.

At this stage, we will state our main result.
\begin{theorem}
Let $u$ be a locally bounded local weak solution to the obstacle problem
to \eqref{eq:pde} with $q \in (0,\infty)$ and structure conditions \eqref{eq:structure}
and a H\"older continuous obstacle function
$\psi \in C^{0; \beta, \frac{\beta}{2}}(\Omega_T)$ for some $\beta \in (0,1)$
in the sense of Definition~\ref{d.obstacle_weaksol}.
Then $u$ is locally H\"older continuous.
\end{theorem}

The theory of H\"older continuity for porous medium type equations is well developed in the obstacle-free case. The first proof goes back to DiBenedetto and Friedman~\cite{DF}, in which non-negative solutions in the degenerate case were considered. The proof in the singular case can be found in~\cite{DGV} and for the treatment of signed solutions we mention~\cite{Liao}.
For more recent developments we refer to~\cite{BDG,BonforteSimonov,LiaoSystem,Mizuno}.
However, for the obstacle problem the theory is not complete yet. H\"older continuity was proven for quasilinear problems in~\cite{StruweVivaldi} and for problems with quadratic growth in~\cite{Choe}. In the case of porous medium type equations H\"older continuity for non-negative solutions to the obstacle problem has been treated in the recent papers~\cite{BLS_holder} and~\cite{Cho_Scheven}. The former concerns the degenerate case for the PME, and the latter the singular case for more general equations with structural conditions analogous to~\eqref{eq:structure}. However, especially the theory for signed solutions is missing, which we are addressing in this paper.

In our proof we use a similar strategy as in~\cite{BLS_holder} and~\cite{Cho_Scheven} for non-negative obstacles, which relies on a De Giorgi type iteration argument. The idea is to construct a sequence of cylinders shrinking to a common vertex. In each of these cylinders we consider measure theoretic alternatives, which we will call first and second alternative.
From these alternatives pointwise information for the solution can be deduced, cf.~Section~\ref{s.1alternative} in case of the first and in Section~\ref{s.2alternative} in case of the second alternative. In particular, we then show in Section~\ref{s.red-osc} that when passing to the subsequent cylinder in the sequence, the oscillation of the solution is reduced by a fixed amount. In the heart of the De Giorgi type iteration are energy estimates for truncations of solutions, which are stated and proved in Section~\ref{s.energyest}. In case of the second alternative, we also exploit the logarithmic estimates from Section~\ref{s.logest}. When deriving suitable energy or logarithmic estimates for a solution to the obstacle problem, additional attention has to be paid when using the comparison function depending on the solution itself. One has to guarantee that it is admissible, especially that it is sufficiently regular in time and stays above the given obstacle. While the energy estimate for truncations from above takes a similar form as in the obstacle-free case, the obstacle will play a role in the estimate for truncations from below. In the De Giorgi type iteration argument, this will be taken into account in the upper bound for the oscillation of $u$. Namely, such an upper bound should be sufficiently large compared to the oscillation of the obstacle.

In order to balance the inhomogeneous scaling behavior of the PME, we will work in cylinders which respect the intrinsic geometry of the equation. In particular, we will use cylinders of the form 
$$
Q_{\rho,\theta \rho^2} (x_o,t_o) := B_\rho (x_o) \times (t_o - \theta \rho^2, t_o),
$$
in which the scaling parameter $\theta$ is comparable to $|u|^{q-1}$. In contrast to the proof for the singular equations in~\cite{Cho_Scheven}, we treat both degenerate and singular cases with cylinders taking the same form (as in the obstacle-free case~\cite{Liao}). 

We will separate two different cases in the proof: when the solution $u$ is near zero, and when $u$ is away from zero. In the latter case, the equation behaves like a linear one. These cases are divided by comparing $u$ with its oscillation. Additional challenges in the case of signed solutions are given by the fact that when $u$ is near zero, the sign of $u$ may change in the cylinder considered. Especially, when applying the second alternative we use a technical argument to be able to avoid the set where $u$ becomes degenerate/singular. Furthermore, when dealing with the case where $u$ is negative and bounded away from zero, additional care is needed in the construction of cylinders.

We also point out that there is an alternative approach to the obstacle problem, in which the solution is defined as the smallest weak supersolution lying above the given obstacle $\psi$. This approach allows to consider fairly irregular obstacles as in~\cite{KLS}. For a study of the connection between these two different notions of solutions we refer to~\cite{AvLu}.

It would be interesting to obtain regularity up to the boundary when suitable boundary values are prescribed.
However, this is a topic for further research, since different techniques are required.

\medskip

\noindent
{\bf Acknowledgments.} K.~Moring has been supported by the Magnus Ehrnrooth Foundation and Foundation for Aalto University Science and Technology.

\section{Definitions and auxiliary results} \label{s.aux}
In order to give a formal definition of solutions, we consider the class of functions
$$
	K_\psi(\Omega_T) :=
	\left\{ v \in C^0((0,T);L_{\loc}^{q+1}(\Omega)):
	v \in L^2_{\loc}(0,T; H^1_{\loc}(\Omega)), v \geq \psi \text{ a.e. in } \Omega_T \right\}.
$$
Admissible comparison maps will be contained in the class of functions
$$
	K'_\psi(\Omega_T) :=
	\left\{ v \in K_\psi(\Omega_T): \partial_t v \in L^{q+1}_{\loc}(\Omega_T) \right\}.
$$

\begin{definition}\label{d.obstacle_weaksol}
We say that $u\in K_\psi(\Omega_T)$ is a local weak solution to the obstacle problem
associated with \eqref{eq:pde} if and only if
\begin{equation}
	\llangle \partial_t ( |u|^{q-1} u ), \varphi (v - u) \rrangle
	+ \iint_{\Omega_T} \mathbf{A}(x,t,u,\nabla u)
	\cdot \nabla \left( \varphi(v - u) \right) \, \d x \d t \geq 0
	\label{eq:weaksol}
\end{equation}
holds true for all comparison maps $v \in K'_\psi(\Omega_T)$
and every test function $\varphi \in C_0^\infty(\Omega_T;\R_{\geq 0})$.
The time term above is defined as
\begin{align*}
	\llangle \partial_t (|u|^{q-1} u), \varphi &(v - u)\rrangle \\
	:=
	&\iint_{\Omega_T} \left[ \partial_t \varphi \left( \frac{q}{q+1} |u|^{q+1}
	- |u|^{q-1} u v \right) - \varphi |u|^{q-1} u \partial_t v \right] \, \d x \d t.
\end{align*}
\end{definition}

\begin{remark}
Note that for $0<q< \frac{n+2}{(n-2)_+}$ every local weak solution is locally bounded. The proof follows the lines of~\cite{BLS_holder} in the degenerate case
and~\cite{Cho_Scheven} in the singular case.
\end{remark} 

For $z_o = (x_o,t_o) \in \Omega_T$, we will work with cylinders of the form
$$
Q_{\rho,s} (z_o) := B_{\rho}(x_o) \times (t_o-s, t_o).
$$
For $b \in \R$ and $\alpha > 0$ we denote the signed $\alpha$-power of $b$ by
\[ \boldsymbol{b}^\alpha :=
\begin{cases}
|b|^{\alpha -1} b &\text{if } b \neq 0,\\
0 &\text{if } b = 0. 
\end{cases}\]

We will exploit the following mollification in time.
For $v \in L^1(\Omega_T)$ and $h>0$, define
\begin{align*}
\mollifytime{v}{h}(x,t)
	:= \frac 1 h \int_0^t e^{\frac{s-t}{h}} v(x,s)\,  \mathrm{d}s.
\end{align*}

We collect some useful properties of the mollification in the following lemma, see~\cite[Lemma 2.9]{Kinnunen-Lindqvist} and~\cite[Appendix B]{BDM:pq}. 
\begin{lemma} \label{lem:mollifier}
Let $v$ and $\mollifytime{v}{h}$ be as above and $p \geq 1$. Then the following properties hold:
\begin{itemize}
\item[(i)]
If $v \in L^p(\Omega_T)$, then 
\begin{equation*}
\mollifytime{v}{h} \to v\quad \text{ in } L^p(\Omega_T) \text{ as } h \to 0.  
\end{equation*}
\item[(ii)]
Let $v\in L^p(0,T; W^{1,p}(\Omega))$. Then
\begin{equation*}
\mollifytime{v}{h}\to v\quad \text{ in } L^p(0,T;W^{1,p}(\Omega)) \text{ as } h \to 0.
\end{equation*}
\item[(iii)]
If $v \in C^0(\overline{\Omega_T})$ and $\Omega \subset \R^n$ is a bounded set, then 
$$
\mollifytime{v}{h} \to v \quad \text{ uniformly in }\Omega_T\ \text{ as }\ h \to 0.
$$
\item[(iv)]
The weak time derivative $\partial_t \mollifytime{v}{h}$ exists in $\Omega_T$ and is given by the formula
\begin{equation*}
  \partial_t \mollifytime{v}{h} = \frac{1}{h} ( v - \mollifytime{v}{h} ).
\end{equation*}
\end{itemize}
\end{lemma}

In this section we recall some standard results needed in the proofs. We begin with a special case of the Sobolev inequality, cf.~\cite[Chapter I, Proposition 3.1]{DB}.
\begin{lemma} \label{l.sobolev}
Let $v \in L^2(0,T; W^{1,2}_0(\Omega))$. Then there exists $c = c(n) > 0$ such that 
\begin{equation*}
\iint_{\Omega_T} |v|^{\frac{2(n+2)}{n}} \, \d x \d t \leq c \left( \iint_{\Omega_T} \left| \nabla v \right|^2 \, \d x \d t \right) \left( \esssup_{0<t<T} \int_{\Omega} |v|^2 \, \d x \right)^\frac{2}{n}.
\end{equation*}
\end{lemma}

We also make use of De Giorgi's isoperimetric inequality and so called fast geometric convergence~\cite[Chapter I, Lemma 2.2]{DB},~\cite[Lemma 7.1]{Giusti}, which we state next.

\begin{lemma} \label{l.isoperimetric_ineq}
Let $k<l$ be real numbers and $B_\rho(x_o) \subset \R^n$. Then for any $v\in W^{1,1}(B_\rho(x_o))$ there exists a constant $c = c(n) > 0$ such that 
$$
(l-k) \left| B_\rho(x_o) \cap \{ v>l \} \right| \leq \frac{c \rho^{n+1}}{\left| B_\rho(x_o) \cap \{ v<k \} \right|} \int_{B_\rho(x_o) \cap \{ k<v<l \}} \left| \nabla v \right| \, \d x
$$
\end{lemma}

\begin{lemma} \label{l.fgc}
Suppose that $\{Y_i\}_{i\in \N_0}$ is a sequence of positive real numbers that satisfy
$$
Y_{i+1} \leq CB^i Y_i^{1+ \sigma}\quad \text{ for all } i \geq 0,
$$
with constants $C,\sigma > 0$ and $B>1$. Then $Y_i \to 0$ as $i \to \infty$ whenever
$$
Y_0 \leq C^{- \frac{1}{\sigma}} B^{- \frac{1}{\sigma^2}}.
$$
\end{lemma}

\section{Energy estimates} \label{s.energyest}
For $w,k \in \R$ let us define 
\begin{align*}
\mathfrak{g}_\pm(w,k) := \pm q \int_k^w  |s|^{q-1} (s-k)_\pm \, \d s.
\end{align*}

The following estimates follow from the definition above, see e.g.~\cite[Lemma 2.2]{BDL}.

\begin{lemma}\label{est:calg}
There exists a constant $c = c(q) > 0$ such that for all $w,k \in \R$ and $q > 0$, the inequality
$$
\frac{1}{c} \left( |w| + |k| \right)^{q-1} (w-k)^2_\pm \leq \mathfrak g_\pm(w,k) \leq c \left( |w| + |k| \right)^{q-1} (w-k)^2_\pm
$$
holds true.
\end{lemma}
Next, we give energy estimates for weak solutions to the obstacle problems.
Note that in the estimate involving $(u-k)_+$ only levels $k$ larger than
the obstacle function are admissible,
whereas there is no restriction on the admissible levels
in the estimate involving $(u-k)_-$.
\begin{lemma} \label{l.caccioppoli}
Let $z_o = (x_o,t_o) \in \Omega_T$ and $Q_{\rho,s}(z_o) \Subset \Omega_T$.
Further, for $\psi \in C^0(\Omega_T)$ we assume that $u \in K_\psi(\Omega_T)$ is a local weak solution
to the obstacle problem to \eqref{eq:pde} with structure conditions \eqref{eq:structure}
in the sense of Definition \ref{d.obstacle_weaksol}.
Then, for any function $\varphi \in C^\infty(Q_{\rho,s}(z_o);\R_{\geq 0})$ vanishing on the lateral boundary of $Q_{\rho,s}(z_o)$ the following estimates hold.
\begin{itemize}
\item[(i)] For all $k \geq \sup_{Q_{\rho,s}(z_o)} \psi$, we have
\begin{align*}
	\max \bigg\{
	&\esssup_{t_o -s < t < t_o} \int_{B_\rho(x_o) \times \{t\}} \varphi^2 \mathfrak{g}_{+}(u,k) \, \d x,
	\iint_{Q_{\rho,s}(z_o)} \varphi^2 |\nabla (u-k)_{+}|^2 \, \d x \d t
	\bigg\} \\
	&\leq
	c \iint_{Q_{\rho,s}(z_o)} [(u-k)_{+}^2|\nabla \varphi|^2 + \mathfrak{g}_{+} (u,k) |\partial_t \varphi^2|] \, \d x \d t \\
	&\phantom{=} + \int_{B_\rho(x_o) \times \{t_o -s\}} \varphi^2 \mathfrak{g}_{+}(u,k)\, \d x.
\end{align*}
\item[(ii)] For arbitrary $k \in \mathbb{R}$, we have
\begin{align*}
	\max \bigg\{
	&\esssup_{t_o -s < t < t_o} \int_{B_\rho (x_o) \times \{t\}} \varphi^2 \mathfrak{g}_{-}(u,k)\, \d x,
	\iint_{Q_{\rho,s}(z_o)} \varphi^2 |\nabla (u-k)_{-}|^2 \, \d x \d t
	\bigg\} \\
	&\leq
	c \iint_{Q_{\rho,s}(z_o)} [(u-k)_{-}^2|\nabla \varphi|^2 + \mathfrak{g}_{-} (u,k) |\partial_t \varphi^2|] \, \d x \d t \\
	&\phantom{=} + \int_{B_\rho(x_o) \times \{t_o -s\}} \varphi^2 \mathfrak{g}_{-}(u,k) \, \d x.
\end{align*}
\end{itemize}
\end{lemma}
\begin{proof}
In the following, we omit $z_o$ for simplification.
We first prove (i).
Let $\varphi \in C^\infty(Q_{\rho,s};\R_{\geq 0})$ vanish on the lateral boundary
$\partial B_\rho \times (-s,0)$ of $Q_{\rho,s}$.
Further, we define $\xi_\varepsilon \in W_0^{1,\infty}([-s,0];[0,1])$ by
\begin{align*}
    \xi_\varepsilon(t) := \begin{cases}
    0, &\text{ for } -s \leq t \leq t_1 - \varepsilon,\\
    1 + \frac{t-t_1}{\varepsilon}, &\text{ for } t_1-\varepsilon < t \leq t_1,\\
    1, &\text{ for } t_1 < t < t_2,\\
    1 - \frac{t-t_2}{\varepsilon}, &\text{ for } t_2 \leq t < t_2 + \varepsilon,\\
    0, &\text{ for } t_2+\varepsilon \leq t \leq 0.
    \end{cases}
\end{align*}
and use $\eta := \varphi^2 (\xi_\varepsilon)_\delta$, in which $(\xi_\varepsilon)_\delta$ is a standard mollification of $\xi_\eps$ with $0<\delta< \tfrac{\eps}{2}$, as a test function in \eqref{eq:weaksol}. Moreover, we define
$$
	w_h :=
	\mollifytime{u}{h} - \big( \mollifytime{u}{h} - k \big)_+
	+ \big\| \psi - \mollifytime{\psi}{h} \big\|_{L^\infty(Q_{\rho,s})}.
$$
Note that there hold
$w_h \in C^0((0,T);L^{q+1}_{\loc}(\Omega)) \cap L^2_{\loc}(0,T;H^1_{\loc}(\Omega))$
and
$$
	\partial_t w_h
	=
	\left\{
	\begin{array}{ll}
		\tfrac1h \big( u - \mollifytime{u}{h} \big)
		&\text{in }  \big\{ \mollifytime{u}{h} \leq k \big\}, \\[5pt]
		0 &\text{otherwise }
	\end{array}
	\right\}
	\in L^{q+1}_{\loc}(\Omega_T)
$$
by the chain rule and Lemma \ref{lem:mollifier}.
Further, we have that $w_h \geq \psi$ in $Q_{\rho, s}$,
since $\mollifytime{u}{h} \geq \mollifytime{\psi}{h}$ in $\Omega_T$
and $k \geq \sup_{Q_{\rho,s}} \psi$.
Since it is sufficient that $w_h$ satisfies the obstacle condition
in $\spt(\eta) \subset Q_{\rho,s}$, $w_h$ is an admissible comparison map
in \eqref{eq:weaksol}.
Therefore, we obtain that
\begin{equation}
	\llangle \partial_t \u^q, \eta (w_h - u)\rrangle
	+ \iint_{\Omega_T} \mathbf{A}(x,t,u,\nabla u) \cdot \nabla (\eta (w_h-u)) \, \d x \d t
	\geq 0.
	\label{eq:weaksol_energy_est}
\end{equation}
In the following, we treat the two terms separately.
First, by the formula for $\partial_t w_h$ above and since
$\big( \u^q - \boldsymbol{\mollifytime{u}{h}^q} \big) \big( u - \mollifytime{u}{h} \big) \geq 0$,
we conclude that
\begin{align*}
    \iint_{Q_{\rho,s}} \eta \u^q \partial_t w_h \, \d x \d t
    \geq
    \iint_{Q_{\rho,s}} \eta \big| \mollifytime{u}{h} \big|^{q-1} \mollifytime{u}{h} 
    \partial_t \big( \mollifytime{u}{h} - \big( \mollifytime{u}{h} - k \big)_+ \big)
    \, \d x \d t.
\end{align*}
Integrating by parts leads to
\begin{align*}
	\iint_{Q_{\rho,s}} &\eta \u^q \partial_t w_h \, \d x \d t \\
	&\geq
	\iint_{Q_{\rho,s}} -\tfrac{1}{q+1} \partial_t \eta \big| \mollifytime{u}{h} \big|^{q+1}
	+ \partial_t \eta \big| \mollifytime{u}{h} \big|^{q-1} \mollifytime{u}{h} \big( \mollifytime{u}{h} -k \big)_+
	\, \d x \d t \\
	&\phantom{=}+
	\iint_{Q_{\rho,s}}
	\eta \partial_t \big( \big| \mollifytime{u}{h} \big|^{q-1} \mollifytime{u}{h} \big)
	\big( \mollifytime{u}{h} -k \big)_+ \, \d x \d t.
\end{align*}
Writing the last term on the right-hand side of the preceding inequality as
$$
	\partial_t \big( |\mollifytime{u}{h}|^{q-1} \mollifytime{u}{h} \big)
	\big( \mollifytime{u}{h}-k \big)_+
	=
	\partial_t \left( q \int_k^{\mollifytime{u}{h}}|s|^{q-1}(s-k)_+\, \d s \right)
	=
	\partial_t \mathfrak{g}_+ \big( \mollifytime{u}{h},k \big)
$$
yields
\begin{align*}
	\iint_{Q_{\rho,s}} &\eta \u^q \partial_t w_h \, \d x \d t \\
	&\geq
	-\iint_{Q_{\rho,s}}\partial_t \eta
	\Big[ \tfrac{1}{q+1} \big| \mollifytime{u}{h} \big|^{q+1}
	- \big| \mollifytime{u}{h} \big|^{q-1} \mollifytime{u}{h}
	\big(\mollifytime{u}{h} - k \big)_+ \Big]
	\, \d x \d t \\
	&\phantom{=}
	-\iint_{Q_{\rho,s}} \partial_t \eta \mathfrak{g}_+\big( \mollifytime{u}{h},k \big)
	\, \d x \d t.
\end{align*}
Inserting this into the first term of \eqref{eq:weaksol_energy_est}
and passing to the limit $h \downarrow 0$, we see that
\begin{equation}
	\limsup_{h \downarrow 0} \,\llangle \partial_t \u^q, \eta (w_h - u) \rrangle
	\leq
	\iint_{Q_{\rho,s}} \partial_t \eta \mathfrak{g}_+(u,k) \, \d x \d t.
	\label{eq:time_part}
\end{equation}
In order to treat the second term in \eqref{eq:weaksol_energy_est},
observe that
$$
	\nabla (\eta (w_h - u)) \longrightarrow -\nabla (\eta (u-k)_+)
	\text{ in $L^2(Q_{\rho,s})$ as $h \downarrow 0$}.
$$
Since $\mathbf{A}(x,t,u,\nabla u) \in L^2(Q_{\rho,s},\R^n)$
by growth condition \eqref{eq:structure}$_2$, this implies
\begin{align}
    \lim_{h \downarrow 0}
    \iint_{Q_{\rho,s}} &\mathbf{A}(x,t,u,\nabla u) \cdot \nabla(\eta (w_h -u)) \, \d x \d t \nonumber\\
    &=
    -\iint_{Q_{\rho,s}} 2 (\xi_\varepsilon)_\delta \varphi (u-k)_+ \mathbf{A}(x,t,u,\nabla u) \cdot \nabla \varphi \, \d x \d t
    \label{eq:diffusion_term} \\
    &\phantom{=}
    - \iint_{Q_{\rho,s}} \eta \mathbf{A}(x,t,u,\nabla u) \cdot \nabla (u-k)_+ \, \d x \d t.
    \nonumber
\end{align}
By means of \eqref{eq:structure}$_1$ and since $\eta \geq 0$, we observe that
$$
	\eta \mathbf{A}(x,t,u,\nabla u) \cdot \nabla (u-k)_+
	=
	\eta \mathbf{A}(x,t,u,\nabla (u-k)_+) \cdot \nabla (u-k)_+
	\geq
	C_o \eta |\nabla (u-k)_+|^2.
$$
Further, by \eqref{eq:structure}$_2$ and
Young's inequality with parameter $\frac{C_o}{4 C_1}$, we find that
\begin{align*}
	\big| 2(\xi_\varepsilon)_\delta \varphi &(u-k)_+ \mathbf{A}(x,t,u,\nabla u) \cdot \nabla \varphi \big| \\
	&\leq
	2 C_1 (\xi_\varepsilon)_\delta |\varphi| |\nabla (u-k)_+| (u-k)_+ |\nabla \varphi| \\
	&\leq
	\tfrac{C_o}{2} (\xi_\varepsilon)_\delta \varphi^2 |\nabla (u-k)_+|^2
	+ \tfrac{2 C_1^2}{C_o} (\xi_\varepsilon)_\delta |\nabla \varphi|^2 (u-k)_+^2.
\end{align*}
Inserting the preceding two inequalities into \eqref{eq:diffusion_term},
we obtain that
\begin{align*}
	\lim_{h \downarrow 0}
	\iint_{Q_{\rho,s}} &\mathbf{A}(x,t,u,\nabla u) \cdot \nabla(\eta (w_h -u)) \, \d x \d t \\
	&\leq
	-\frac{C_o}{2}\iint_{Q_{\rho,s}} \eta |\nabla (u-k)_+|^2 \, \d x \d t
	+ \frac{2C_1^2}{C_o}
	\iint_{Q_{\rho,s}} (\xi_\varepsilon)_\delta |\nabla \varphi|^2 (u-k)_+^2 \, \d x \d t.
\end{align*}
Together with \eqref{eq:time_part}, we conclude that
\begin{align*}
    \frac{C_o}{2}\iint_{Q_{\rho,s}} &\eta |\nabla (u-k)_+|^2 \, \d x \d t \\
    &\leq
    \iint_{Q_{\rho,s}} \partial_t \eta \mathfrak{g}_+(u,k) \, \d x \d t
    + \frac{2C_1^2}{C_o} \iint_{Q_{\rho,s}} (\xi_\varepsilon)_\delta |\nabla \varphi|^2(u-k)_+^2 \, \d x \d t.
\end{align*}
By first passing to the limit $\delta \to 0$, and subsequently $\varepsilon \downarrow 0$, we get
\begin{align*}
    \int_{B_\rho \times\{t_2\}} &\varphi^2 \mathfrak{g}_+(u,k) \, \d x
    + \int_{t_1}^{t_2}\int_{B_{\rho} } \varphi^2 |\nabla (u-k)_+|^2 \, \d x \d t \\
    &\leq
    c\iint_{Q_{\rho,s}} |\nabla \varphi|^2 (u-k)_+^2
    + |\partial_t \varphi^2| \mathfrak{g}_+(u,k) \, \d x \d t
    +\int_{B_\rho \times \{t_1\}} \varphi^2 \mathfrak{g}_+(u,k) \, \d x.
\end{align*}
Since all terms are non-negative, we infer the desired energy estimate
by first discarding the first term on the left-hand side and
passing to the limits $t_1 \downarrow -s$, $t_2 \uparrow 0$
and then discarding the second term on the left-hand side,
passing to the limit $t_1 \downarrow -s$
and taking the supremum over all $t_2 \in (-s,0)$.\\
In order to prove (ii), we use the comparison function
$w_h := \mollifytime{u}{h} + (\mollifytime{u}{h} - k)_- + \| \psi - \mollifytime{\psi}{h}\|_{L^\infty(Q_{\rho,s})}$ with an arbitrary level $k \in \mathbb{R}$ and proceed similarly as in (i).
\end{proof}

\section{Logarithmic estimates} \label{s.logest}
In this section we will state a logarithmic estimate as in~\cite{BLS_holder} and~\cite{Cho_Scheven}. Let $0 < \gamma < \Gamma$ and define
$$
\phi(a):= \phi_{\Gamma,\gamma}(a) := \left[ \log \left( \frac{\Gamma}{\Gamma + \gamma - a} \right) \right]_+ \ \quad \text{ for } a< \Gamma + \gamma.
$$
Observe that for $a \leq \Gamma$, we have
$$
0 \leq \phi(a) \leq \log \left( \tfrac{\Gamma}{\gamma} \right) \quad \text{ and } \quad 0 \leq \phi'(a) \leq \tfrac{1}{\gamma}\quad \text{ when } a \neq \gamma,
$$	
with $\phi(a) = 0$ for $a \leq \gamma$. Further, we have $\phi'' (a) = \left( \phi' \right)^2(a)$ for $a \neq \gamma$. Note that $\phi^2$ is differentiable in $[0,\Gamma]$ such that the Lipschitz continuous derivative $\left( \phi^2 \right)'$ satisfies
$$
\left( \phi^2 \right)' = 2 \phi \phi'\quad \text{ and } \quad \left( \phi^2 \right)'' = 2 (1+\phi) 	\left( \phi' \right)^2 \quad \text{ in } [0,\Gamma] \setminus \{ \gamma \}.
$$
With this information at hand, we are able to prove the following lemma.
\begin{lemma} \label{l.logarithmic_est}
Let $B_{\rho_1}(x_o) \Subset B_{\rho_2}(x_o) \Subset \Omega$, $0 < t_1 < t_2 < T$
and $Q_2 := B_{\rho_2}(x_o) \times (t_1,t_2)$.
Further, define $\Gamma := \esssup_{Q_2} (u-k)_+ $ and consider some parameter $\gamma \in (0,\Gamma)$.
Assume that $\psi \in C^0(\Omega_T)$ and let $u \in  K_\psi(\Omega_T)$
be a local weak solution to the obstacle problem \eqref{eq:pde}
with structure conditions \eqref{eq:structure}
according to Definition~\ref{d.obstacle_weaksol}.
Then the following estimates hold.
\begin{enumerate}
\item[(i)] For any $k \geq \sup_{Q_2} \psi$, we have
\begin{align*}
	\sup_{t \in (t_1,t_2)}
	\int_{B_{\rho_1} (x_o)} \int_k^u &|s|^{q-1} \left( \phi^2 \right)'((s - k)_+) \, \d s \d x \\
	&\leq
	\int_{B_{\rho_2}(x_o) \times \{t_1\}} \int_k^u |s|^{q-1} \left( \phi^2 \right)'((s-k)_+) \, \d s \d x \\ 
	&\phantom{=}
	+ \frac{c}{(\rho_2 - \rho_1)^2} \iint_{Q_2} \phi\left( (u-k)_+ \right) \, \d x \d t.
\end{align*}
\item[(ii)] For any $k \in \R$, we have that
\begin{align*}
	\sup_{t \in (t_1,t_2)}
	\int_{B_{\rho_1} (x_o)} \int_u^k &|s|^{q-1} \left( \phi^2 \right)'((s - k)_-) \, \d s \d x \\
	&\leq
	\int_{B_{\rho_2}(x_o) \times \{t_1\}} \int_u^k |s|^{q-1} \left( \phi^2 \right)'((s-k)_-) \, \d s \d x \\ 
	&\phantom{=}
	+ \frac{c}{(\rho_2 - \rho_1)^2} \iint_{Q_2} \phi\left( (u-k)_- \right) \, \d x \d t.
\end{align*}
\end{enumerate}
\end{lemma}

\begin{proof}
In the following, we omit $(x_o,t_o)$ for simplicity.
We start with the proof of (i). Since all terms in the asserted estimate depend continuously on $k$,
we may assume that $k > \sup_{Q_2} \psi$.
We would like to use
$$
	w_h :=
	\mollifytime{u}{h}
	- \lambda \left( \phi^2 \right)' \big( \big(\mollifytime{u}{h} - k \big)_+ \big)
	+ \big\| \psi - \mollifytime{\psi}{h} \big\|_{L^\infty(Q_2)}
$$
with
$$
	0 < \lambda \leq
	\tfrac{ ( k - \sup_{Q_2}\psi ) }{\sup_{[0,\Gamma]} \left( \phi^2 \right)' }
$$
as comparison map in \eqref{eq:weaksol}.
By Lemma \ref{lem:mollifier} and since $\big( \phi^2 \big)'$ is
a Lipschitz continuous function, we have that $w_h \in C^0((0,T); L^{q+1}_{\loc}(\Omega))
\cap L^2_{\loc}(0,T; H^1_{\loc}(\Omega))$ with $\partial_t w_h \in L^{q+1}_{\loc}(\Omega_T)$.
Moreover, if $\mollifytime{u}{h} \leq k$, we find that
$$
	w_h =
	\mollifytime{u}{h} +\big \| \psi - \mollifytime{\psi}{h} \big\|_{L^\infty(Q_2)}
	\geq
	\mollifytime{\psi}{h} +\big \| \psi - \mollifytime{\psi}{h} \big\|_{L^\infty(Q_2)}
	\geq \psi
$$
in $Q_2$, and for $\mollifytime{u}{h} > k$ we have that
$$
	w_h >
	k - \lambda \left( \phi^2 \right)' \big( \big( \mollifytime{u}{h} - k \big)_+ \big) 
	\geq
	\sup_{Q_2} \psi
$$
by the restriction on $\lambda$.
Consequently, $w_h$ is an admissible comparison map in \eqref{eq:weaksol}.
Thus, for any $\varphi \in C^\infty_0(Q_2; \R_{\geq 0})$ we obtain that 
$$
	\mathrm{I}_h + \mathrm{II}_h
	:=
	\llangle \partial_t \u^q, \varphi (w_h - u) \rrangle
	+ \iint_{\Omega_T} \mathbf{A}(x,t,u, \nabla u) \cdot \nabla \left(\varphi (w_h-u) \right) \, \d x \d t 
	\geq 0.
$$
In the following, we estimate these terms separately.
First, we calculate
\begin{align}
	\partial_t w_h
	&=
	\partial_t \mollifytime{u}{h}
	\left( 1- \lambda \big(\phi^2 \big)''\big( \big( \mollifytime{u}{h}-k \big)_+ \big) \right) \label{eq:aux_time_logest} \\
	&=
	\tfrac{1}{h} \big( u - \mollifytime{u}{h} \big)
	\left(1- \lambda \big(\phi^2 \big)'' \big( \big( \mollifytime{u}{h}-k \big)_+ \big) \right).
	\nonumber
\end{align}
Since the derivative $\partial_t \mollifytime{u}{h}$ vanishes a.e.~in the set
$\big\{ \big( \mollifytime{u}{h}-k \big)_+ = \gamma \big\}$,
the terms involving $\big( \phi^2 \big)''$ are well-defined a.e.~in $Q_2$.
Further, decreasing $\lambda$ if necessary, we may assume that the last factor
is positive, which allows us to estimate
$$
	\big( \u^q - \boldsymbol{\mollifytime{u}{h}^q} \big) \partial_t w_h
	=
	\tfrac{1}{h}
	\big( \u^q - \boldsymbol{\mollifytime{u}{h}^q} \big) \big( u - \mollifytime{u}{h} \big)
	\left(1- \lambda \big(\phi^2 \big)'' \big( \big( \mollifytime{u}{h}-k \big)_+ \big) \right)
	\geq 0.
$$
Together with integration by parts and the fact that
$$
	\partial_t \boldsymbol{\mollifytime{u}{h}^q} \lambda
	\big( \phi^2 \big)' \big( \big( \mollifytime{u}{h}-k \big)_+ \big)
	=
	\partial_t \left(\lambda q
	\int_k^{\mollifytime{u}{h}} |s|^{q-1} \big( \phi^2 \big)'((s-k)_+) \, \d s \right),
$$
this implies that
\begin{align*}
	\iint_{\Omega_T} &\varphi \u^q \partial_t w_h \, \d x \d t \\
	&\geq
	\iint_{\Omega_T} \varphi \boldsymbol{\mollifytime{u}{h}^q} \partial_t w_h \, \d x \d t \\
	&=
	\iint_{\Omega_T}
	\varphi \partial_t \Big( \tfrac{1}{q+1} \left|\mollifytime{u}{h} \right|^{q+1} \Big)
	- \lambda \varphi \boldsymbol{\mollifytime{u}{h}^q}
	\partial_t \left( \big( \phi^2 \big)' \big( \big( \mollifytime{u}{h}-k \big)_+ \big) \right)
	\,\dx\dt
	\\
	&=
	\iint_{\Omega_T} \partial_t \varphi
	\left( - \tfrac{1}{q+1}  \big| \mollifytime{u}{h} \big|^{q+1}
	+ \boldsymbol{\mollifytime{u}{h}^q} \lambda
	\left( \phi^2 \right)' \big( \big( \mollifytime{u}{h}-k \big)_+ \big) \right) \, \d x \d t \\
	&\phantom{=}
	- q \lambda \iint_{\Omega_T} \partial_t \varphi
	\int_k^{\mollifytime{u}{h}} |s|^{q-1} \left( \phi^2 \right)'\left( (s-k)_+ \right) \,\ds \, \d x \d t.
\end{align*}
Recalling the definition of $\mathrm{I}_h$ and inserting the preceding inequality yields
\begin{equation}
	\limsup_{h \to 0} \mathrm{I}_h
	\leq
	q \lambda
	\iint_{\Omega_T} \partial_t \varphi \int_k^u |s|^{q-1} \big( \phi^2 \big)'((s-k)_+)\, \d s  \, \d x \d t.
	\label{eq:time_term_logest}
\end{equation}
Next, we turn our attention to $\mathrm{II}_h$.
By Lemma \ref{lem:mollifier} and since $(\phi^2)'$ is Lipschitz continuous, we find that
$$
	\varphi (w_h - u) \rightharpoonup - \lambda \left( \phi^2 \right)'( (u-k)_+ ) \varphi
	\quad \text{weakly in } L^2(t_1,t_2;W^{1,2}(B_{\rho_1})).
$$
Together with the structure conditions \eqref{eq:structure} this implies that
\begin{align*}
	\lim_{h \to 0} \mathrm{II}_h
	&=
	- \lambda \iint_{\Omega_T} \mathbf{A}(x,t,u,\nabla u) \cdot \nabla \left(\varphi \left( \phi^2 \right)'\left( (u-k)_+ \right) \right) \, \d x \d t \\
	&=
	-\lambda \iint_{\Omega_T} \varphi \big( \phi^2 \big)''( (u-k)_+ )
	\mathbf{A}(x,t,u,\nabla u) \cdot \nabla (u-k)_+ \,\dx\dt \\
	&\phantom{=}
	- \lambda \iint_{\Omega_T} \big( \phi^2 \big)' ((u-k)_+)
	\mathbf{A}(x,t,u,\nabla u) \cdot \nabla \varphi \, \d x \d t \\
	&\leq
	-\lambda C_o \iint_{\Omega_T} \varphi \big( \phi^2 \big)''( (u-k)_+ )
	|\nabla (u-k)_+ |^2 \,\dx\dt \\
	&\phantom{=}
	+ \lambda C_1 \iint_{\Omega_T} \big( \phi^2 \big)' ((u-k)_+)
	|\nabla (u-k)_+ | |\nabla \varphi| \,\dx\dt. 
\end{align*}
Here, the term involving $\big( \phi^2 \big)''$ is well defined a.e.~in $\Omega_T$,
since $\nabla (u-k)_+ = 0$ a.e.~in $\{ (u-k)_+ = \gamma \}$.
At this point we choose $\varphi(x,t) = \xi_\varepsilon(t) \eta(x)^2$,
where $\xi_\varepsilon$ is defined as in the proof of Lemma \ref{l.caccioppoli}
and $\eta \in C_0^1(B_{\rho_2},\R_{\geq 0})$ is a cut-off function
with $\eta = 1$ in $B_{\rho_1}$ and $|\nabla \eta| \leq \frac{2}{\rho_2-\rho_1}$.
Applying Young's inequality with parameter $\frac{C_o}{2C_1}$ yields 
\begin{align*}
	\lim_{h\to 0} \mathrm{II}_h
	&\leq
	\lambda C_o \iint_{\Omega_T} \xi_\varepsilon \eta^2
	\Big( 2\phi \big(\phi'\big)^2 - \big( \phi^2 \big)'' \Big)((u-k)_+)
	|\nabla (u-k)_+|^2 \,\dx\dt\\
	&\phantom{=}
	+ \tfrac{2\lambda C_1}{C_o}
	\iint_{\Omega_T} \xi_\varepsilon |\nabla \eta|^2 \phi((u-k)_+) \, \d x \d t \\
	&\leq
	\tfrac{2\lambda C_1}{C_o}
	\iint_{\Omega_T} \xi_\varepsilon |\nabla \eta|^2 \phi((u-k)_+) \, \d x \d t.
\end{align*}
In the last line, we used that
$2\phi \left(\phi'\right)^2 - \left( \phi^2 \right)'' = - 2 \left( \phi' \right)^2 \leq 0$.
Together with \eqref{eq:aux_time_logest}, we obtain that
\begin{align*}
	-q \iint_{\Omega_T} \xi_\varepsilon' \eta^2
	&\int_k^u |s|^{q-1} \big( \phi^2 \big)'((u-k)_+) \, \d s \, \d x \d t \\
	&\leq
	 \tfrac{2 C_1}{C_o} \iint_{\Omega_T} \xi_\varepsilon |\nabla \eta|^2 \phi((u-k)_+) \, \d x \d t.
\end{align*}
Passing to the limit $\varepsilon \downarrow 0$, we conclude that 
\begin{align*}
	q \int_{B_{\rho_1}\times \{ t \}} &\int_k^u |s|^{q-1} \left( \phi^2 \right)'((u-k)_+) \, \d s \, \d x \\
	&\leq
	q \int_{B_{\rho_2} \times \{t_1\}} \int_k^u |s|^{q-1} \left( \phi^2 \right)'((u-k)_+) \, \d s \, \d x \\
	&\phantom{=}
	+ \tfrac{8C_1}{C_o(\rho_2 - \rho_1)^2} \iint_{Q_2} \phi((u-k)_+) \, \d x \d t,
\end{align*}
for any $t \in (t_1,t_2)$, which proves (i).

For the case (ii) we start with the comparison function
\begin{align*}
w_h &:= \mollifytime{u}{h} + \left( \phi^2 \right)' \left( \left(\mollifytime{u}{h} - k\right)_- \right)  + \| \psi - \mollifytime{\psi}{h} \|_{L^\infty(Q_2)} \\
&\geq \psi + \left( \phi^2 \right)' \left( (u-k)_- \right) \\
&\geq \psi
\end{align*}
since $\phi, \phi' \geq 0$ and proceed similarly as in the case (i).
\end{proof}

\section{First alternative} \label{s.1alternative}
In the following, we use parameters $\boldsymbol{\mu}^-, \boldsymbol{\mu}^+ \in \R$ and $\boldsymbol{\omega} >0$ satisfying
\begin{equation}
	\boldsymbol{\mu}^- \leq \essinf_{Q_{\rho, \theta \rho^2}(z_o)} u,
	\quad
	\boldsymbol{\mu}^+ \geq \esssup_{Q_{\rho, \theta \rho^2}(z_o)} u
	\quad \text{ and }\quad
	\boldsymbol{\omega} = \boldsymbol{\mu}^+ - \boldsymbol{\mu}^-,
	\label{eq:mu_basic}
\end{equation}
where slightly different factors $\theta \approx |u|^{q-1}$ will be considered.
For some constant $\nu \in (0,1)$, distinguish between the measure theoretic alternatives
\begin{equation}
	\left\{
	\begin{array}{l}
		\big| \big\{ \pm \big( \boldsymbol{\mu}^\pm - u \big) \leq \tfrac12 \boldsymbol{\omega} \big\} \cap Q_{\rho, \theta \rho^2}(z_o) \big|
	\leq \nu | Q_{\rho, \theta \rho^2}(z_o)|, \\[5pt]
	\big| \big\{ \pm \big( \boldsymbol{\mu}^\pm - u \big) \leq \tfrac12 \boldsymbol{\omega} \big\} \cap Q_{\rho, \theta \rho^2}(z_o) \big|
	> \nu | Q_{\rho, \theta \rho^2}(z_o)|.
	\end{array}
	\right.
	\label{eq:explanation_alternatives}
\end{equation}
In the so-called first alternative~\eqref{eq:explanation_alternatives}$_1$, the solution $u$ is bounded away from its essential infimum or supremum on a large portion of the considered cylinder, whereas it is close to the essential infimum or supremum on a large part of the cylinder in the second alternative~\eqref{eq:explanation_alternatives}$_2$.
In both situations, our goal is to show that $u$ is bounded away from one of the extreme values a.e.~in a suitable sub-cylinder of $Q_{\rho, \theta \rho^2}(z_o)$ together with a quantitative bound.
The necessary tools for the first alternative~\eqref{eq:explanation_alternatives}$_1$ will be discussed in the present section, while we will be concerned with the second alternative~\eqref{eq:explanation_alternatives}$_2$ in Section~\ref{s.2alternative}.

\subsection{De Giorgi type lemmas in the singular/degenerate case}
Next we will state and prove a De Giorgi type lemma in the case where $u$ is away from its supremum in a significant portion of the considered intrinsic cylinder.

\begin{lemma} \label{l.de_giorgi_type+++}
Let $u$ be a locally bounded, local weak solution to the obstacle problem and $Q_{\rho,\theta \rho^2} (z_o) \Subset \Omega_T$, where $\theta = \boldsymbol{\omega}^{q-1}$. Furthermore, we suppose that $\tfrac12 (\boldsymbol{\mu}^+ + \boldsymbol{\mu}^-) \geq \sup_{Q_{\rho,\theta \rho^2}(z_o)} \psi$. Then, there exists a constant $\nu = \nu(n,q) \in (0,1)$, such that if
$$
\left|\{ \boldsymbol{\mu}^+ - u \leq  \tfrac12 \boldsymbol{\omega}\} \cap Q_{\rho,\theta \rho^2}(z_o)\right| \leq \nu \left|Q_{\rho,\theta \rho^2}(z_o)\right|,
$$
then either
$$
|\boldsymbol{\mu}^+| > 2 \boldsymbol{\omega},
$$
or
$$
\boldsymbol{\mu}^+ - u \geq  \tfrac{1}{4}\boldsymbol{\omega}\quad \text{ a.e. in } Q_{\frac{\rho}{2},\theta \left( \frac{\rho}{2}\right)^2}(z_o).
$$
\end{lemma}

\begin{proof}
We omit the fixed reference point $(x_o,t_o)$ for simplicity and assume that
$|\boldsymbol{\mu}^+| \leq 2 \boldsymbol{\omega}$.
Observe that by Lemma~\ref{est:calg} there holds
$$
\mathfrak{g}_+ (u,k) \leq c \left( |u| + |k| \right)^{q-1} \left( u-k \right)_+^2 \leq c \left( |u| + |k| \right)^q \left( u-k \right)_+
$$
for any $k \in \R$.
Further, for $\tilde{k} > k$ we have $(u-k)_+ \geq (u- \tilde{k})_+ $.
From the energy estimate, Lemma~\ref{l.caccioppoli} (i), we obtain
\begin{align*}
\esssup_{-\theta \rho^2<t<0} &\int_{B_\rho} \varphi^2 \left( |u| + |k| \right)^{q-1} ( u-\tilde{k} )_+^2\, \d x  + \iint_{Q_{\rho,\theta \rho^2}} \varphi^2 \left| \nabla (u-\tilde{k})_+ \right|^2 \, \d x \d t \\
&\leq c \iint_{Q_{\rho,\theta \rho^2}} (u-k)_+^2 |\nabla \varphi|^2 \, \d x \d t + c \iint_{Q_{\rho,\theta \rho^2}} (|u| + |k|)^q(u-k)_+ |\partial_t \varphi^2| \, \d x \d t
\end{align*}
for any $\tilde{k} > k \geq \sup_{Q_{\rho,\theta \rho^2}} \psi$ and $\varphi \in C^\infty(Q_{\rho,\theta \rho^2}; \R_{\geq 0})$ vanishing on the parabolic boundary of $Q_{\rho,\theta \rho^2}$.
For $j \in \N_0$ we choose
\begin{align*} 
    \begin{cases}
    k_j = \boldsymbol{\mu}^+ - \frac{\boldsymbol{\omega}}{4} - \frac{\boldsymbol{\omega}}{2^{j+2}},& \,\, \tilde k_j = \frac{k_j + k_{j+1}}{2},\\
    \rho_j = \frac{\rho}{2} + \frac{\rho}{2^{j+1}},&\,\, \tilde \rho_j = \frac{\rho_j + \rho_{j+1}}{2},\\
    B_j = B_{\rho_j},& \,\, \widetilde B_j = B_{\tilde \rho_j},\\
Q_j = Q_{\rho_j,\theta \rho_j^2},& \,\, \widetilde Q_j = Q_{\tilde \rho_j,\theta \tilde \rho_j^2}.    
    \end{cases}
\end{align*}
Note that $k_j \geq k_0 = \frac{1}{2} (\boldsymbol{\mu}^+ + \boldsymbol{\mu}^-) \geq \sup_{Q_0} \psi \geq \sup_{Q_j} \psi$ for any $k \in \N_0$, since $Q_{\rho,\theta \rho^2} = Q_0 \supset Q_1 \supset \ldots$ and by the assumption on $k_0$.
Furthermore, we use a smooth cut-off function $0 \leq \varphi \leq 1$ vanishing on the parabolic boundary of $Q_j$ and equal to identity in $\widetilde{Q}_j$ such that 
$$
|\nabla \varphi| \leq c \frac{2^j}{\rho}\quad \text{ and } \quad |\partial_t \varphi| \leq c \frac{2^{2j}}{\theta \rho^2}.
$$
Moreover, we set $A_j = \left\{ u > k_j \right\} \cap Q_j$.
Formulating the preceding energy estimate for these quantities and using that the assumption $|\boldsymbol{\mu}^+| \leq 2 \boldsymbol{\omega}$ implies  that $5 \boldsymbol{\omega} \geq |u| + |k_j| \geq u - k_j  \geq  \tilde{k}_j - k_j = 2^{-(j+4)} \boldsymbol{\omega}$ on the set where $u \geq \tilde{k}_j$, we find that
\begin{align*}
	\min\big\{ &5^{q-1}, 2^{-(q-1)(j+4)} \big\} \boldsymbol{\omega}^{q-1} \esssup_{-\theta \tilde{\rho}_j^2<t<0} \int_{\widetilde{B}_{j}} ( u-\tilde{k}_j )_+^2\, \d x
	+ \iint_{\widetilde{Q}_j} |\nabla (u-\tilde{k}_j)_+|^2 \, \d x \d t \\
	&\leq
	\esssup_{-\theta \tilde{\rho}_j^2<t<0}
	\int_{\widetilde{B}_{j}} \varphi^2 \left( |u| + |k_j| \right)^{q-1} ( u-\tilde{k}_j )_+^2\, \d x 
	+ \iint_{\widetilde{Q}_{j}} \varphi^2 \left| \nabla (u-\tilde{k}_j)_+ \right|^2 \, \d x \d t \\
	&\leq
	c\frac{2^{2j}}{\rho^2}\iint_{Q_j} (u-k_j)_+^2 \, \d x \d t + c \frac{2^{2j}}{ \theta \rho^2} \iint_{Q_j} (|u| + |k_j|)^q(u-k_j)_+  \, \d x \d t \\
	&\leq
	c \frac{2^{2j}}{\rho^2}\boldsymbol{\omega}^2 \left| A_j \right|.
\end{align*}
In particular, the estimate above gives us that
\begin{equation} \label{e.supestimate+}
	\esssup_{-\theta \tilde{\rho}_j^2<t<0} \int_{\widetilde{B}_{j}} ( u-\tilde{k}_j )_+^2\, \d x
	\leq
	c \frac{2^{(2 + (q-1)_+)j}}{\rho^2} \boldsymbol{\omega}^{3-q} |A_j|
\end{equation}
holds true.
By introducing a smooth cut-off function $0 \leq \phi \leq 1$, such that $\phi$ equals identity in $Q_{j+1}$ and vanishes on the lateral boundary of $\widetilde Q_j$ with $|\nabla \phi |\leq c 2^j \rho^{-1}$, H\"older's and the Sobolev embedding from Lemma~\ref{l.sobolev} inequality, we infer
\begin{align*}
	&\frac{\boldsymbol{\omega}}{2^{j+4}} |A_{j+1}|
	\leq
	\iint_{\widetilde Q_j} (u-\tilde k_j)_+ \phi  \, \d x \d t \\
	&\leq
	\left( \iint_{\widetilde Q_j} \left[ (u- \tilde k_j)_+ \phi \right] ^{2 \frac{n+2}{n}} \, \d x \d t\right)^\frac{n}{2(n+2)}
	\left| A_j \right|^{1- \frac{n}{2(n+2)}}  \\
	&\leq
	c \left( \iint_{\widetilde{Q}_j} \left|\nabla\left[ (u-\tilde{k}_j)_+ \phi  \right] \right|^2\, \d x \d t\right)^\frac{n}{2(n+2)}
	\left( \esssup_{-\theta \tilde{\rho}_j^2<t<0} \int_{\widetilde{B}_j} (u-\tilde{k}_j)_+^2 \, \d x \right)^\frac{1}{n+2} \\
	&\phantom{=}
	\cdot |A_{j}|^{1-\frac{n}{2(n+2)}} \\
	&\leq
	c \left( \frac{2^{2j}}{\rho^2} \boldsymbol{\omega}^2 |A_j| \right)^\frac{n}{2(n+2)}
	\left( \frac{2^{(2 + (q-1)_+)j}}{\rho^2} \boldsymbol{\omega}^{3-q} |A_j| \right)^\frac{1}{n+2}
	|A_{j}|^{1-\frac{n}{2(n+2)}} \\
	&\leq
	c\, 2^{\left( 1 + \frac{(q-1)_+}{n+2} \right) j} \rho^{-1} \boldsymbol{\omega}^{1 + \frac{1-q}{n+2}} |A_j|^{1+\frac{1}{n+2}}.
\end{align*}
Here we also used energy estimate once more and~\eqref{e.supestimate+}. This implies that
$$
	|A_{j+1}| \leq
	c\, 2^{\left( 2 + \frac{(q-1)_+}{n+2} \right)j} \rho^{-1} \boldsymbol{\omega}^{\frac{1-q}{n+2}} |A_j|^{1+\frac{1}{n+2}}.
$$
By dividing this by $|Q_{j+1}|$ and denoting $Y_j = |A_j|/|Q_j|$ we have
\begin{align*}
	Y_{j+1}
	&\leq
	c\, 2^{\left( 2 + \frac{(q-1)_+}{n+2} \right)j} \rho^{-1} \boldsymbol{\omega}^{\frac{1-q}{n+2}} \frac{|Q_j|^{1+\frac{1}{n+2}}}{|Q_{j+1}|} Y_j^{1+\frac{1}{n+2}} \\
	&\leq
	c \, 2^{\left( 2 + \frac{(q-1)_+}{n+2} \right)j} Y_j^{1+\frac{1}{n+2}}.
\end{align*}
Then, if $\nu \leq c^{-(n+2)} B^{-(n+2)^2}$, where $B = 2^{2 + \frac{(q-1)_+}{n+2}}$ we may use Lemma~\ref{l.fgc} to conclude the proof.
\end{proof}

We state the De Giorgi type lemma for the case where $u$ is away from its infimum without proof.
However, defining $k_j = \boldsymbol{\mu}^- + \tfrac{\boldsymbol{\omega}}{4} + \tfrac{\boldsymbol{\omega}}{2^{j+2}}$ and exploiting the energy estimate in Lemma~\ref{l.caccioppoli} (ii), it can be proven analogous to Lemma~\ref{l.de_giorgi_type+++}.
Observe that in this case any level is admissible in the energy estimate.

\begin{lemma} \label{l.de_giorgi_type---}
Let $u$ be a locally bounded, local weak solution to the obstacle problem and $Q_{\rho,\theta \rho^2}(z_o) \Subset \Omega_T$ with $\theta = \boldsymbol{\omega}^{q-1}$. Then, there exists a constant $\nu \in (0,1)$ depending on the data such that if
$$
|\{ u - \boldsymbol{\mu}^- \leq \tfrac12 \boldsymbol{\omega}\} \cap Q_{\rho,\theta \rho^2}(z_o)| \leq \nu |Q_{\rho,\theta \rho^2} (z_o)|,
$$
then either
$$
|\boldsymbol{\mu}^-| > 2 \boldsymbol{\omega},
$$
or
$$
u - \boldsymbol{\mu}^- \geq \tfrac{1}{4}\boldsymbol{\omega}\quad \text{a.e. in }  Q_{\frac{\rho}{2},\theta \left(\frac{\rho}{2}\right)^2}(z_o).
$$
\end{lemma}

\subsection{De Giorgi type lemmas in the non-degenerate/non-singular case}
In this section, we state De Giorgi type lemmas for the case where $u$ is bounded away from its supremum and the case where $u$ is bounded away from its infimum.
However, we only prove the former one, since the proof of the latter is analogous.
For the first lemma, we suppose that
\begin{equation} \label{e.away_zero_+}
\boldsymbol{\mu}^+ < - \tfrac14 \boldsymbol{\omega}
\end{equation}
holds true. Observe that this is equivalent to
\begin{equation*}
	\boldsymbol{\mu}^+ < \tfrac{1}{5} \boldsymbol{\mu}^-.
\end{equation*}
In the first lemma we use the scaling 
\begin{equation} \label{e.scaling_awayzero-}
\left| \boldsymbol{\mu}^- \right| \leq \theta^\frac{1}{q-1} \leq  \left|5 \boldsymbol{\mu}^-\right|.
\end{equation}

\begin{lemma} \label{l.de_giorgi_awayzero+}
Assume that \eqref{e.away_zero_+} and \eqref{e.scaling_awayzero-} hold and let $u$ be a locally bounded, local weak solution to the obstacle problem and $Q_{\rho,\theta \rho^2}(z_o) \Subset \Omega_T$. Furthermore, suppose that $\tfrac12 \left( \boldsymbol{\mu}^+ + \boldsymbol{\mu}^- \right) \geq \sup_{Q_{\rho,\theta \rho^2}(z_o)} \psi$. Then, there exists a constant $\nu \in (0,1)$ such that if
$$
\left|\{ \boldsymbol{\mu}^+ - u \leq \tfrac12 \boldsymbol{\omega} \} \cap  Q_{\rho,\theta \rho^2}(z_o)\right| \leq \nu \left|Q_{\rho,\theta \rho^2}(z_o)\right|,
$$
then
$$
\boldsymbol{\mu}^+ - u \geq \tfrac{1}{4}\boldsymbol{\omega}\quad \text{ a.e. in }  Q_{\frac{\rho}{2},\theta \left(\frac{\rho}{2}\right)^2}(z_o).
$$
\end{lemma}

\begin{proof}
For simplicity we omit the reference point $z_o$ in the notation. Observe that 
$$
\mathfrak{g}_+ (u,k) \approx \left( |u| + |k| \right)^{q-1} \left( u-k \right)_+^2
\quad \text{for any } k \in \R
$$
up to a constant depending only on $q$ by Lemma \ref{est:calg}.
Further, for $\tilde{k} > k$ we have that $(u-k)_+ \geq (u- \tilde{k})_+ $. From the energy estimate, Lemma~\ref{l.caccioppoli} (i), we obtain
\begin{align*}
&\esssup_{-\theta \rho^2<t<0} \int_{B_\rho} \varphi^2 \left( |u| + |k| \right)^{q-1} ( u-\tilde{k})_+^2\, \d x  + \iint_{Q_{\rho,\theta \rho^2}} \varphi^2 \left| \nabla (u-\tilde{k})_+ \right|^2 \, \d x \d t \\
&\leq c \iint_{Q_{\rho,\theta \rho^2}} (u-k)_+^2 |\nabla \varphi|^2 \, \d x \d t + c \iint_{Q_{\rho,\theta \rho^2}} (|u| + |k|)^{q-1} (u-k)_+^2 |\partial_t \varphi^2| \, \d x \d t
\end{align*}
for $\tilde{k} > k \geq \sup_{Q_{\rho,\theta \rho^2}} \psi$ and $\varphi \in C^\infty(Q_{\rho,\theta \rho^2}; \R_{\geq 0})$ vanishing on the parabolic boundary of $Q_{\rho,\theta \rho^2}$.
For $k \in \N_0$, we choose $k_j$, $\tilde{k}_j$, $\rho_j$, $\tilde{\rho}_j$,
$B_j$, $\widetilde B_j$, $Q_j$, $\widetilde Q_j$ and the cut-off function $\varphi$ as in the proof of Lemma~\ref{l.de_giorgi_type+++}.
However, here \eqref{e.away_zero_+} in particular implies that $k_j < \boldsymbol{\mu}^+ < 0$ for any $j \in \N_0$.
Thus, we know that $2 \left| \boldsymbol{\mu}^+ \right| \leq 2|u| \leq |u|+ |k_j| \leq 2 |k_j|$ in $A_j := \{ u > k_j \} \cap Q_j$, which gives us that
\begin{align*}
\min &\left\{ \left| \boldsymbol{\mu}^+ \right|^{q-1}, |k_j|^{q-1} \right\} \esssup_{-\theta \tilde{\rho}_j^2<t<0} \int_{\widetilde{B}_{j}} ( u-\tilde{k}_j )_+^2\, \d x + \iint_{\widetilde{Q}_j} |\nabla (u-\tilde{k}_j)_+|^2 \, \d x \d t \\
	&\leq
\esssup_{-\theta \tilde{\rho}_j^2<t<0} \int_{\widetilde{B}_{j}} \varphi^2 \left( |u| + |k_j| \right)^{q-1} ( u-\tilde{k}_j )_+^2\, \d x  + \iint_{\widetilde{Q}_{j}} \varphi^2 | \nabla (u-\tilde{k}_j)_+ |^2 \, \d x \d t \\
&\leq c\frac{2^{2j}}{\rho^2}\iint_{Q_j} (u-k_j)_+^2 \, \d x \d t + c \frac{2^{2j}}{\theta \rho^2} \iint_{Q_j} (|u| + |k_j|)^{q-1}(u-k_j)_+^2  \, \d x \d t \\
&\leq c \frac{2^{2j} \boldsymbol{\omega}^2}{\rho^2}\left(1 + \max\left\{ \left| \boldsymbol{\mu}^+ \right|^{q-1}, |k_j|^{q-1} \right\} \theta^{-1} \right) \left| A_j \right|.
\end{align*}
Observe that the bounds $\frac{1}{5} |\boldsymbol{\mu}^-| < |\boldsymbol{\mu}^+| < |\boldsymbol{\mu}^-|$ and $|\boldsymbol{\mu}^+| < |k_j| < |\boldsymbol{\mu}^-|$
hold true by \eqref{e.away_zero_+} and the definition of $k_j$.
Taking also \eqref{e.scaling_awayzero-} into account, we estimate the preceding inequality further.
In particular, we conclude that
$$
	\esssup_{-\theta \tilde{\rho}_j^2<t<0} \int_{\widetilde{B}_{j}} ( u-\tilde{k}_j )_+^2\, \d x
	\leq
	\frac{c \boldsymbol{\omega}^2}{\theta} \frac{2^{2j}}{\rho^2} |A_j|.
$$
Next, note that $u - k_j \geq \tilde{k}_j -  k_j  = 2^{-(j+4)} \boldsymbol{\omega}$ if $u \geq \tilde{k}_j$.
Using H\"older's inequality and the Sobolev inequality from Lemma \ref{l.sobolev}, we find that
\begin{align*}
	\frac{\boldsymbol{\omega}}{2^{j+4}} &|A_{j+1}| \\
	&\leq
	c \left( \iint_{\widetilde{Q}_j} |\nabla [ (u-\tilde{k}_j)_+ \phi ] |^2\, \d x \d t\right)^\frac{n}{2(n+2)}
	\left( \esssup_{-\theta \tilde{\rho}_j^2<t<0} \int_{\widetilde{B}_j} (u-\tilde{k}_j)_+^2 \, \d x \right)^\frac{1}{n+2} \\
	&\phantom{=}
	\cdot |A_{j}|^{1-\frac{n}{2(n+2)}} \\
	&\leq
	c \left( \frac{2^{2j} \boldsymbol{\omega}^2}{\rho^2} |A_j| \right)^\frac{n}{2(n+2)} 
	\left( \frac{2^{2j} \boldsymbol{\omega}^2}{\theta \rho^2} |A_j| \right)^\frac{1}{n+2} |A_{j}|^{1-\frac{n}{2(n+2)}} \\
	&\leq
	c \, 2^{j} \boldsymbol{\omega} \theta^{-\frac{1}{n+2}} \rho^{-1} |A_j|^{1+\frac{1}{n+2}}.
\end{align*}
This implies that
$$
	|A_{j+1}| \leq
	c \, 2^{2j} \theta^{-\frac{1}{n+2}} \rho^{-1} |A_j|^{1+\frac{1}{n+2}}.
$$
Dividing the preceding inequality by $|Q_{j+1}|$ and denoting $Y_j = |A_j|/|Q_j|$ we have that
\begin{align*}
	Y_{j+1}
	&\leq
	c \, \frac{2^{2j}}{\theta^\frac{1}{n+2} \rho} \frac{|Q_j|^{1+\frac{1}{n+2}}}{|Q_{j+1}|} Y_j^{1+\frac{1}{n+2}}
	\leq
	c \, 2^{2j} Y_j^{1+\frac{1}{n+2}}.
\end{align*}
Thus, we are able to conclude the proof by using Lemma~\ref{l.fgc}.
\end{proof}


Next, we suppose that
\begin{equation} \label{e.away_zero_-}
\boldsymbol{\mu}^- > \tfrac14 \boldsymbol{\omega}
\end{equation}
holds true, which implies
$$
0 \leq \tfrac14 \boldsymbol{\omega} < \boldsymbol{\mu}^- \leq \boldsymbol{\mu}^+ < 5 \boldsymbol{\mu}^-.
$$
Further, we use the scaling 
\begin{equation} \label{e.scaling_awayzero+}
\tfrac12 \boldsymbol{\mu}^+ \leq \theta^\frac{1}{q-1} \leq 5 \boldsymbol{\mu}^+.
\end{equation}

\begin{lemma} \label{l.de_giorgi_awayzero-}
Consider $Q_{\rho,\theta \rho^2} \Subset \Omega_T$, assume that \eqref{e.away_zero_-} and \eqref{e.scaling_awayzero+} hold and let $u$ be a locally bounded, local weak solution to the obstacle problem. Then, there exists a constant $\nu \in (0,1)$ such that if
$$
\left| \{ u - \boldsymbol{\mu}^- \leq \tfrac12 \boldsymbol{\omega} \} \cap Q_{\rho,\theta \rho^2}(z_o) \right| \leq \nu \left|Q_{\rho,\theta \rho^2}(z_o)\right|,
$$
then
$$
u - \boldsymbol{\mu}^- \geq \tfrac{1}{4}\boldsymbol{\omega}\quad \text{ a.e. in } Q_{\frac{\rho}{2},\theta \left( \frac{\rho}{2}\right)^2}(z_o).
$$
\end{lemma}

\section{Second alternative} \label{s.2alternative}
\subsection{Second alternative near infimum}
Suppose that $\boldsymbol{\mu}^+, \boldsymbol{\mu}^-$ and $\boldsymbol{\omega}$ are given by \eqref{eq:mu_basic}.
In this section we assume that either
\begin{equation}
	\label{e.mu-_bounds_nearzero}
	- \tfrac14 \boldsymbol{\omega} \leq \boldsymbol{\mu}^+ \leq \tfrac12 \boldsymbol{\omega}
	\quad \text{ and }\quad
	\theta = \boldsymbol{\omega}^{q-1}
\end{equation}
holds true, or 
\begin{equation} \label{e.mu-_bounds_awayzero}
\boldsymbol{\mu}^+ < - \tfrac14 \boldsymbol{\omega}\quad \text{ and }\quad \left| \boldsymbol{\mu}^- \right| \leq \theta^\frac{1}{q-1} \leq \left| 5 \boldsymbol{\mu}^- \right|.
\end{equation}
Since $\boldsymbol{\mu}^+ = \boldsymbol{\mu}^- + \boldsymbol{\omega}$, \eqref{e.mu-_bounds_nearzero}$_1$ is equivalent to
\begin{equation*} 
	- \tfrac54 \boldsymbol{\omega} \leq \boldsymbol{\mu}^- \leq - \tfrac12 \boldsymbol{\omega}.
\end{equation*}
Further, observe that~\eqref{e.mu-_bounds_awayzero}$_1$ implies
$$
0 > - \tfrac14 \boldsymbol{\omega} > \boldsymbol{\mu}^+ \geq \boldsymbol{\mu}^- > 5 \boldsymbol{\mu}^+.
$$
First we prove an auxiliary lemma. 
\begin{lemma} \label{l.de_giorgi2_scheven2}
Let $Q_{\rho,\theta \rho^2}(z_o) \Subset \Omega_T$ be a parabolic cylinder and $\nu \in (0,1)$ and $\eta \in \big(0,\frac{1}{8}\big]$.
Assume that \eqref{e.mu-_bounds_nearzero} or \eqref{e.mu-_bounds_awayzero} holds and that $u$ is a locally bounded, local weak solution to the obstacle problem.
Then there exists $\nu_1 = \nu_1(n,q,C_o,C_1,\nu) \in (0,1)$ such that if
$$
\left| \left\{ u - \boldsymbol{\mu}^- < \eta \boldsymbol{\omega} \right\}
\cap Q_{\rho,\frac12 \nu \theta \rho^2} \right|< \nu_1 \left| Q_{\rho,\frac12 \nu \theta \rho^2} \right|, 
$$
then
$$
u - \boldsymbol{\mu}^- \geq \tfrac12 \eta \boldsymbol{\omega}\quad \text{ a.e. in } Q_{\frac{\rho}{2},\frac12 \nu \theta \left(\frac{\rho}{2}\right)^2}
$$
\end{lemma}

\begin{proof}
We omit $z_o$ for simplicity and start the proof by defining
$$
	\begin{array}{c}
		k_j = \boldsymbol{\mu}^- + \bigg( \frac{\eta}{2} + \frac{\eta}{2^{j+1}} \bigg) \boldsymbol{\omega},
		\tilde{k}_j = \tfrac{k_j + k_{j+1}}{2},
		\rho_j = \frac{\rho}{2} + \frac{\rho}{2^{j+1}},
		\tilde{\rho}_j = \tfrac{\rho_j + \rho_{j+1}}{2}, \\
		B_j = B_{\rho_j},
		\widetilde{B}_j = B_{\tilde{\rho}_j},
		Q_j := Q_{\rho_j, \frac12 \nu\theta \rho_j^2},
		\widetilde{Q}_j := Q_{\tilde{\rho}_j, \frac12 \nu\theta \tilde{\rho}_j^2}.
	\end{array}
$$
Observe that by~\eqref{e.mu-_bounds_nearzero} and~\eqref{e.mu-_bounds_awayzero} it follows that $k_j < 0$ for all $j \in \N_0$.
Let $0 \leq \varphi \leq 1$ be a cut-off function that equals identity in $\widetilde Q_j$ and vanishes on the parabolic boundary of $Q_j$ such that
$$
\left| \nabla \varphi \right| \leq c\frac{2^j}{\rho}\quad \text{ and }\quad \left| \partial_t \varphi \right|\leq c \frac{2^{2j}}{\nu \theta \rho^2}.
$$
From the fact that $\boldsymbol{\mu}^- \leq u < k_j < 0$ in the set $A_j := \{ u < k_j \} \cap Q_j$, Lemma~\ref{est:calg} and the energy estimate in Lemma~\ref{l.caccioppoli} (ii), we then obtain
\begin{align*}
	\min&\left\{ \left| \boldsymbol{\mu}^- \right|^{q-1}, |k_j|^{q-1} \right\}
	\esssup_{-\theta \tilde\rho_j^2<t<0} \int_{\widetilde{B}_j} \varphi^2 ( u-\tilde{k}_j)_-^2\, \d x 
	+ \iint_{\widetilde{Q}_j} \varphi^2 \left| \nabla (u-\tilde{k}_j)_- \right|^2 \, \d x \d t \\
	&\leq
	\esssup_{-\theta \tilde\rho_j^2<t<0} \int_{\widetilde{B}_j} \varphi^2 \left( |u| + |k_j| \right)^{q-1} ( u-\tilde{k}_j)_-^2\, \d x 
	+ \iint_{\widetilde{Q}_j} \varphi^2 \left| \nabla (u-\tilde{k}_j)_- \right|^2 \, \d x \d t \\
	&\leq
	c \iint_{Q_j} (u-k_j)_-^2 |\nabla \varphi|^2 \, \d x \d t + c \iint_{Q_j} (|u| + |k_j|)^{q-1} (u-k_j)_-^2 |\partial_t \varphi^2| \, \d x \d t \\
	&\leq
	c \frac{2^{2j}}{\rho^2} \left( 1 + \max \left\{ \left| \boldsymbol{\mu}^- \right|^{q-1}, |k_j|^{q-1} \right\} (\nu \theta)^{-1} \right) \iint_{Q_j} (u-k_j)_-^2 \, \d x \d t \\
	&\leq 
	c \frac{2^{2j} \left( \eta \boldsymbol{\omega}\right)^2}{\rho^2} \left( 1 + \max \left\{ \left| \boldsymbol{\mu}^- \right|^{q-1}, |k_j|^{q-1} \right\} \theta^{-1} \right) \left| A_j\right|
\end{align*}
with a constant $c = c(n,q,C_o,C_1,\nu)$.
On the one hand if \eqref{e.mu-_bounds_nearzero} is satisfied, we know that $\frac{1}{2} \boldsymbol{\omega} \leq |\boldsymbol{\mu}^-| \leq \frac{5}{4} \boldsymbol{\omega}$.
Further, we have that $\frac{3}{8} \boldsymbol{\omega} \leq |\boldsymbol{\mu}^- + \frac{1}{8} \boldsymbol{\omega}| \leq |k_j| \leq |\boldsymbol{\mu}^-| \leq \frac{5}{4} \boldsymbol{\omega}$.
Thus, $|\boldsymbol{\mu}^-|^{q-1}$ and $|k_j|^{q-1}$ are comparable to $\theta = \boldsymbol{\omega}^{q-1}$.
On the other hand, if~\eqref{e.mu-_bounds_awayzero} holds true, we have that $\theta \approx | \boldsymbol{\mu}^-|^{q-1}$ and $\frac{1}{2} |\boldsymbol{\mu}^-| \leq |\boldsymbol{\mu}^- + \frac{1}{8} \boldsymbol{\omega}| \leq |k_j| \leq |\boldsymbol{\mu}^-|$.
With these estimates at hand, we infer in particular
$$
\esssup_{-\theta \tilde\rho_j^2<t<0} \int_{\widetilde{B}_j} \varphi^2 ( u-\tilde{k}_j)_-^2\, \d x \leq c \frac{2^{2j} \left( \eta \boldsymbol{\omega}\right)^2}{\theta \rho^2} \left| A_j\right|.
$$
Let $0 \leq \phi \leq 1$ be a cut-off function that equals identity in $Q_{j+1}$ and vanishes outside of $\widetilde Q_j$.
Since $k_j - u \geq k_j - \tilde k_j = 2^{-(j+3)} \eta \boldsymbol{\omega}$ in the set $\{u \leq \tilde k_j\}$, by H\"older's inequality and the Sobolev embedding from Lemma \ref{l.sobolev} we obtain that
\begin{align*}
	\frac{\eta \boldsymbol{\omega}}{2^{j+3}} &\left| A_{j+1} \right|
	\leq
	\iint_{\widetilde{Q}_j}  (u- \tilde k_j )_- \phi \, \d x \d t \\
	&\leq
	c \left( \iint_{\widetilde Q_j} [ (u-\tilde k_j)_- \phi ]^\frac{2(n+2)}{n} \, \d x \d t \right)^\frac{n}{2(n+2)} \left| A_j \right|^{1- \frac{n}{2(n+2)}} \\
	&\leq
	\left( \iint_{\widetilde Q_j} | \nabla [ (u-\tilde k_j)_- \phi ] |^2 \, \d x \d t \right)^\frac{n}{2(n+2)}
	\left( \esssup_{-\theta \tilde \rho_j^2 < t < 0} \int_{\widetilde B_j} (u- \tilde k_j)_-^2 \, \d x \right)^\frac{1}{n+2} \\
	&\phantom{=}
	\cdot \left| A_j \right|^{1- \frac{n}{2(n+2)}} \\
	&\leq
	c \left( \frac{2^{2j} (\eta \boldsymbol{\omega})^2 }{\rho^2} \right)^\frac{n}{2(n+2)} \left( \frac{2^{2j} (\eta \boldsymbol{\omega})^2}{\theta \rho^2} \right)^\frac{1}{n+2} \left| A_j \right|^{1+\frac{1}{n+2}} \\
	&=
	c \frac{2^j \eta \boldsymbol{\omega}}{\theta^\frac{1}{n+2} \rho} \left| A_j \right|^{1+ \frac{1}{n+2}}.
\end{align*}
Dividing by $|Q_{j+1}|$ and denoting $Y_j = |A_j| / |Q_j|$, we conclude that
$$
Y_{j+1} \leq c 2^{2j} Y_{j}^{1+ \frac{1}{n+2}}
$$
for a constant $c = c(n,q,C_o,C_1,\nu)$. Setting $\nu_1 \leq c^{-(n+2)} 4^{-(n+2)^2}$,
we conclude the proof by using Lemma~\ref{l.fgc}.
\end{proof}

At this stage, we state the main result in this section, which allows us to deal with arbitrary $\nu$ in the assumed measure estimate.
In contrast, in the preceding lemma $\nu_1$ is a fixed constant depending only on the data.
\begin{lemma} \label{l.de_giorgi_2_alternative2}
Let $Q_{2\rho,\theta (2\rho)^2}(z_o) \Subset \Omega_T$ be a parabolic cylinder.
Assume that \eqref{e.mu-_bounds_nearzero} or \eqref{e.mu-_bounds_awayzero} holds and that $u$ is a locally bounded, local weak solution to the obstacle problem.
Then for any $\nu \in (0,1)$ there exists a constant $a = a(n,q,C_o,C_1,\nu) \in \big(0,\frac{1}{64}\big]$ such that if
$$
\left| \{ \boldsymbol{\mu}^+ - u \leq \tfrac12 \boldsymbol{\omega} \} \cap Q_{\rho,\theta \rho^2}(z_o)  \right| > \nu \left| Q_{\rho,\theta \rho^2}(z_o) \right|,
$$
then 
$$
u -\boldsymbol{\mu}^- \geq a \boldsymbol{\omega}\quad \text{ a.e. in } Q_{\frac{\rho}{2},\frac12 \nu \theta \left( \frac{\rho}{2} \right)^2}(z_o).
$$
\end{lemma}

\begin{proof}
In the following, we omit $z_o$ for simplicity.
Observe that from the assumption it follows that 
\begin{equation} \label{e.measure_estimate_q2alt}
\left| \left\{ u - \boldsymbol{\mu}^- < \tfrac12 \boldsymbol{\omega} \right\} \cap Q_{\rho,\theta\rho^2}  \right| < (1-\nu) \left| Q_{\rho,\theta\rho^2} \right|,
\end{equation}
which further implies 
\begin{equation} \label{e.measure_estimate_b2alt}
\left| \left\{ u (\cdot, t_1) - \boldsymbol{\mu}^- < \tfrac12 \boldsymbol{\omega} \right\}\cap B_\rho \right| \leq \frac{1-\nu}{1-\frac12 \nu} \left| B_\rho \right|
\end{equation}
for some $t_1 \in [- \theta \rho^2, - \frac12 \nu \theta \rho^2]$. If this did not hold, we would have
\begin{align*}
	\left| \{ u - \boldsymbol{\mu}^- < \tfrac12 \boldsymbol{\omega} \} \cap Q_{\rho,\theta\rho^2} \right|
	&\geq
	\int_{-\theta \rho^2}^{- \frac12 \nu \theta \rho^2}
	\left| \{ u(\cdot,t) - \boldsymbol{\mu}^- < \tfrac12 \boldsymbol{\omega} \} \cap B_\rho \right| \,  \d t \\
&\geq (1- \tfrac12 \nu) \theta \rho^2 \frac{1-\nu}{1- \tfrac12 \nu} \left| B_\rho \right|\\
&= (1- \nu) \left| Q_{\rho, \theta \rho^2} \right|,
\end{align*}
which contradicts~\eqref{e.measure_estimate_q2alt}. Now we want to propagate the measure information in~\eqref{e.measure_estimate_b2alt} to the whole interval $(t_1,0)$ by using the logarithmic estimate.
Let us define the level $k$ by
$$
k:= \boldsymbol{\mu}^- + \delta \boldsymbol{\omega},
$$
where $\delta \in \big(0,\frac{1}{8} \big]$. This implies 
\begin{equation*}
u < 0 \quad \text{ in } \{ u < k \}
\end{equation*}
by the bound~\eqref{e.mu-_bounds_nearzero}$_1$ or~\eqref{e.mu-_bounds_awayzero}$_1$ depending on the case. Let us choose $s_o \in \N_{\geq 5}$ large enough such that 
$$
s_o > 1 - \tfrac{\log \delta}{\log 2} \geq 4,
$$
which implies that
$$
2^{1-s} < \delta \leq \tfrac{1}{8}
$$
for every $s \geq s_o$.
Let us also suppose that 
\begin{equation} \label{a.inf_u}
\essinf_{B_\rho\times (t_1,0)} u \leq \boldsymbol{\mu}^- + \tfrac{\delta}{2} \boldsymbol{\omega}
\end{equation}
and define
$$
H := \esssup_{B_\rho \times (t_1,0)} (u-k)_-.
$$
Now it follows that $ \tfrac{\delta}{2} \boldsymbol{\omega} \leq H \leq \delta \boldsymbol{\omega} $, which implies
\begin{equation} \label{e.H_bounds}
\tfrac{1}{2^s} \boldsymbol{\omega} \leq H \leq \tfrac{1}{8} \boldsymbol{\omega}\quad \text{ for any } s \geq s_o.
\end{equation}
Further, we define the function 
\begin{equation*} 
\phi (v) := \left[ \log\left( \frac{H}{H + \tfrac{1}{2^s}\boldsymbol{\omega} - v} \right) \right]_+
\end{equation*}
for $v < H + \tfrac{1}{2^s} \boldsymbol{\omega}$.
Now, we rewrite the integrals in Lemma~\ref{l.logarithmic_est} (ii) as
$$
\int_u^k |s|^{q-1} \left( \phi^2 \right)' ((s-k)_-) \, \d s = \int_0^{(u-k)_-} | k - \tau |^{q-1} \left(\phi^2\right)' (\tau) \, \d \tau
$$
and take into account that $|k| \leq |k - \tau| \leq |u|$.
Thus, we deduce that
\begin{align*}
	\min  \big\{ |\boldsymbol{\mu}^-|^{q-1},|k|^{q-1} \big\} \mathrm{I}(t)
	&:=
	\min \big\{ |\boldsymbol{\mu}^-|^{q-1},|k|^{q-1} \big\} \int_{B_{\sigma \rho} \times \{t\}} \phi^2((u - k)_-) \, \d x \\
	&\leq
	\max \big\{ |\boldsymbol{\mu}^-|^{q-1},|k|^{q-1} \big\} \int_{B_{\rho} \times \{t_1\}} \phi^2((u-k)_-) \, \d x \\
	&\phantom{=}
	+ \frac{c}{(1 - \sigma)^2 \rho^2} \iint_{B_\rho \times (t_1,0)} \phi\left( (u-k)_- \right) \, \d x \d t
\end{align*}
for any $t \in (t_1,0)$ and $\sigma \in (0,1)$.
Observe that in both cases \eqref{e.mu-_bounds_nearzero} and \eqref{e.mu-_bounds_awayzero} we may estimate
$\frac{1}{2} |\boldsymbol{\mu}^-| \leq (1 - 4\delta) |\boldsymbol{\mu}^-|
\leq |\boldsymbol{\mu}^- + \delta \boldsymbol{\omega}| = |k| \leq |\boldsymbol{\mu}^-|$.
In addition, if \eqref{e.mu-_bounds_nearzero} holds, we have that $\frac{1}{2} \boldsymbol{\omega} \leq |\boldsymbol{\mu}^-| \leq \frac{5}{4} \boldsymbol{\omega}$ and $\theta = \boldsymbol{\omega}^{q-1}$.
Recall that $\theta \approx |\boldsymbol{\mu}^-|^{q-1}$ in the case \eqref{e.mu-_bounds_awayzero}.
Since $\phi$ is increasing and by~\eqref{e.H_bounds} we also find that
$$
\phi((u-k)_-) \leq \log \left( \frac{2^s H}{\boldsymbol{\omega}} \right) \leq \log \left( 2^{s-3} \right),
$$
which implies
\begin{align*}
	\mathrm{I} (t)
	&\leq
	(1 - 4\delta)^{-|q-1|} \left( \log \left( 2^{s-3} \right) \right)^2 \left| \{ u(\cdot,t_1) < k \} \cap B_\rho \right| + \frac{c \log \left( 2^{s-3} \right)}{\theta(1-\sigma)^2 \rho^2} \left| B_\rho \times (t_1,0) \right| \\
	&\leq
	\left[ (1 - 4\delta)^{-|q-1|} \left( \log \left( 2^{s-3}  \right) \right)^2 \frac{1- \nu}{1 - \tfrac12 \nu} + \frac{c\,\log \left( 2^{s-3} \right)}{(1-\sigma)^2} \right] | B_\rho |
\end{align*}
for any $t\in (t_1,0)$, where we used also the fact that $t_1 \geq -\theta \rho^2$.
On the left hand side, let us consider the set $B_{\sigma \rho} \cap \{ u(\cdot,t) \leq \boldsymbol{\mu}^- + \tfrac{1}{2^s}\boldsymbol{\omega} \}$ for $t \in (t_1,0)$, where
$$
(u-k)_- \geq \boldsymbol{\mu}^- + \delta \boldsymbol{\omega} - \boldsymbol{\mu}^- - \tfrac{1}{2^s} \boldsymbol{\omega} = \left( \delta - \tfrac{1}{2^s} \right) \boldsymbol{\omega}.
$$
Since the function $\phi( (u-k)_- )$ is decreasing in $H$ and $H \leq \delta \boldsymbol{\omega}$, this implies 
\begin{align*}
\phi ( (u-k)_- ) &\geq \left[ \log \left( \frac{\delta \boldsymbol{\omega}}{\delta \boldsymbol{\omega} + \frac{1}{2^s}\boldsymbol{\omega} - (\delta - \frac{1}{2^s})\boldsymbol{\omega}} \right) \right]_+ = \log \left( 2^{s-1} \delta \right).
\end{align*}
Therefore, we find that
\begin{equation*}
\mathrm{I}(t) \geq \left( \log (2^{s-1} \delta) \right)^2 \left| \left\{ u(\cdot,t) - \boldsymbol{\mu}^- \leq \tfrac{1}{2^s} \boldsymbol{\omega} \right\} \cap B_{\sigma\rho}  \right|. 
\end{equation*}
By combining the preceding estimates, we obtain that
\begin{align*}
	\big| \big\{ &u(\cdot,t) - \boldsymbol{\mu}^- \leq \tfrac{1}{2^s} \boldsymbol{\omega} \big\} \cap B_{\sigma \rho} \big| \\
	&\leq
	\frac{1}{\left( \log (2^{s-1} \delta) \right)^2}
	\left[ (1 - 4\delta)^{-|q-1|}\left( \log \left( 2^{s-3} \right) \right)^2 \frac{1- \nu}{1 - \tfrac12 \nu}
	+ \frac{c\, \log \left( 2^{s-3} \right)}{(1-\sigma)^2} \right] \left| B_\rho \right|. 
\end{align*}
Using $\left| B_\rho \setminus B_{\sigma \rho} \right| \leq n (1-\sigma) |B_\rho|$, this yields
\begin{align*}
	&\left|\left\{ u(\cdot,t) - \boldsymbol{\mu}^- \leq \tfrac{1}{2^s} \boldsymbol{\omega} \right\} \cap B_{\rho} \right| \\
	&\leq
	\bigg[ (1-4\delta)^{-|q-1|}
	\bigg(\frac{ \log ( 2^{s-3})}{\log (2^{s-1} \delta)} \bigg)^2 
	\frac{1- \nu}{1 - \tfrac12 \nu}
	+ \frac{c \log( 2^{s-3})}{(1-\sigma)^2 \left( \log (2^{s-1}\delta) \right)^2}
	+ n(1-\sigma) \bigg] \\
	&\phantom{=}\;
	\cdot |B_\rho|. 
\end{align*}
Let us fix
$$
\sigma := 1 - \frac{\nu^2}{8n} \in (0,1),
$$
and $\delta$ such that 
$$
4 \delta = \min \left\{ \frac{1}{2}, 1- \left( \frac{1-\nu^2}{1- \frac12 \nu^2} \right)^\frac{1}{|q-1|} \right\}.
$$
This implies that
$$
\left( 1 - 4 \delta \right)^{- |q-1|} \leq \frac{1- \frac12 \nu^2}{1 - \nu^2}
$$
and we obtain that
\begin{align*}
	&\left| \left\{ u(\cdot,t) - \boldsymbol{\mu}^- \leq \tfrac{1}{2^s} \boldsymbol{\omega} \right\} \cap B_\rho \right| \\
	&\hspace{10mm}\leq \left[ \left( \frac{\log (2^{s-3})}{\log (2^{s-1} \delta)} \right)^2 \frac{1- \tfrac12 \nu^2}{(1+\nu)(1-\tfrac12 \nu)} + c \frac{n^2 \log (2^{s-3})}{\nu^4 \left( \log(2^{s-1}\delta)\right)^2} + \frac{\nu^2}{8} \right] \left| B_\rho \right|
\end{align*}
for all $t \in (t_1,0)$.
Let $s_o \in \N_{\geq 5}$ depending on $n$, $q$, $C_o$, $C_1$ and $\nu$ be so large that 
$$
\left( \frac{\log (2^{s_o-3})}{\log (2^{s_o-1} \delta)} \right)^2 \leq (1+\nu)(1 - \tfrac12\nu).
$$
and 
$$
c \frac{n^2 \log (2^{s_o-3} )}{\left( \log(2^{s_o-1}\delta)\right)^2} \leq \frac{\nu^6}{8}.
$$

Now we conclude that 
\begin{equation*}
\left| \left\{ u(\cdot,t) - \boldsymbol{\mu}^- \leq \tfrac{1}{2^s} \boldsymbol{\omega} \right\} \cap B_\rho \right| \leq \left( 1- \tfrac14 \nu^2 \right) \left| B_\rho \right|,
\end{equation*}
or equivalently
\begin{equation*}
\left| \left\{ u(\cdot,t) - \boldsymbol{\mu}^- > \tfrac{1}{2^s} \boldsymbol{\omega} \right\} \cap B_\rho \right| \geq \tfrac14 \nu^2 \left| B_\rho \right|
\end{equation*}
for all $t \in (t_1,0)$ and $s \geq s_o$.
Observe that we assumed~\eqref{a.inf_u} to obtain the preceding estimate.
However, if~\eqref{a.inf_u} is false, we have that
$$
\left| \{ u(\cdot,t) - \boldsymbol{\mu}^- \leq \tfrac{\delta}{2}\boldsymbol{\omega} \} \cap B_\rho \right| = 0
$$
for all $t \in (t_1,0)$, which implies the previous inequalities since $\tfrac{\delta}{2} > \tfrac{1}{2^s}$ for $s \geq s_o$.

Up next, let $\nu_1 \in (0,1)$ be the parameter from Lemma~\ref{l.de_giorgi2_scheven2}. We show that there exists $s_1 \in \N_{\geq 2}$ such that 
\begin{equation} \label{e.measure_est_s0s1_2}
\left| \left\{ u - \boldsymbol{\mu}^- < \tfrac{1}{2^{s_o + s_1}} \boldsymbol{\omega} \right\} \cap Q_{\rho, \frac12 \nu \theta \rho^2} \right| < \nu_1 \left| Q_{\rho, \frac12 \nu \theta \rho^2} \right|.
\end{equation}
To this end, we define cylinders $Q_2= B_\rho \times \left(- \frac12 \nu \theta \rho^2, 0 \right]$ and $Q_1 = B_\rho \times \left(- \nu \theta \rho^2,0\right]$, which implies that $Q_2 \subset Q_1 \subset Q_{2\rho,\theta (2\rho)^2}$.
Further, we consider levels
$$
k_j := \boldsymbol{\mu}^- + \tfrac{1}{2^j} \boldsymbol{\omega}
$$
and set
$$
A_j := \{ u < k_j \} \cap Q_2
$$
for $j \in \N_{\geq s_o}$. By De Giorgi's isoperimetric inequality from Lemma~\ref{l.isoperimetric_ineq}, we have that
\begin{align*}
	( k_j - k_{j+1} ) &| \{ u(\cdot,t) < k_{j+1} \} \cap B_\rho | \\
	&\leq
	\frac{c(n) \rho^{n+1}}{| \{ u(\cdot,t) > k_j \} \cap B_\rho |}
	\int_{B_\rho \cap \{ k_{j+1} < u(\cdot,t) < k_j \}} | \nabla u | \, \d x  \\
	&\leq
	\frac{c(n) \rho}{\nu^2} \int_{B_\rho \cap \{ k_{j+1} < u(\cdot,t) < k_j \}} | \nabla u | \, \d x.
\end{align*}
Integrating over $(- \tfrac12 \nu \theta \rho^2, 0)$, we find that
\begin{align*}
( k_j - k_{j+1} ) |A_{j+1}| &\leq \frac{c(n) \rho}{\nu^2} \iint_{A_j \setminus A_{j+1}} \left| \nabla u \right| \, \d x \d t \\
&\leq \frac{c(n) \rho}{\nu^2} \left| A_j \setminus A_{j+1} \right|^\frac12 \left( \iint_{A_j \setminus A_{j+1}} |\nabla u|^2 \, \d x \right)^\frac12 \\
&\leq \frac{c(n) \rho}{\nu^2} \left| A_j \setminus A_{j+1} \right|^\frac12 \left( \iint_{Q_2} \left| \nabla (u-k_j)_- \right|^2 \, \d x \d t \right)^\frac12.
\end{align*}
Applying Lemma~\ref{l.caccioppoli} (ii), we estimate the integral on the right-hand side by
\begin{align*}
\iint_{Q_2} | \nabla (u-k_j)_-|^2 \, \dx \d t &\leq c \left( \frac{1}{\rho^2} + \frac{\max \{|k_j|^{q-1}, |\boldsymbol{\mu}^-|^{q-1}\}}{\nu \theta \rho^2} \right) \iint_{Q_1} (u-k_j)_-^2 \, \d x \d t \\
&\leq \frac{c}{\nu \rho^2} \left( \frac{\boldsymbol{\omega}}{2^j} \right)^2 |Q_1|.
\end{align*}
In the last line, we used $|k_j| \geq \tfrac14 \boldsymbol{\omega}$ when $0<q<1$ and $|\boldsymbol{\mu}^-| \leq \tfrac54 \boldsymbol{\omega}$ when $q > 1$ in case~\eqref{e.mu-_bounds_nearzero}, whereas we have that $|k_j| \geq \tfrac12 |\boldsymbol{\mu}^-|$ in case~\eqref{e.mu-_bounds_awayzero} when $0<q<1$.
Combining the two estimates above and using $k_j - k_{j+1} = 2^{-(j+1)} \boldsymbol{\omega}$, we infer
$$
|A_{j+1}|^2 \leq \frac{c}{\nu^5} |A_j \setminus A_{j+1}| |Q_1|.
$$
At this stage, we sum over $j = s_o, \ldots, s_o + s_1 -1$ for some $s_1 \in \N_{\geq 2}$, which gives us
$$
s_1 |A_{s_o+s_1}|^2 \leq \frac{c}{\nu^5} | Q_1 |^2 = \frac{c}{\nu^5} |Q_2|^2.
$$
Choosing $s_1$ large enough, the estimate~\eqref{e.measure_est_s0s1_2} holds true.
Hence, an application of Lemma~\ref{l.de_giorgi2_scheven2} with $\eta = 2^{-(s_o+s_1)}$ yields 
$$
u \geq \boldsymbol{\mu}^- + \tfrac{1}{2^{s_o+s_1+1}} \boldsymbol{\omega}\quad \text{ a.e. in } Q_{\frac{\rho}{2}, \frac12 \nu \theta \left( \frac{\rho}{2}\right)^2}.
$$
By denoting $a = a(n,q,C_o, C_1,\nu) = \frac{1}{2^{s_o+s_1+1}}$ the proof is completed.
\end{proof}

\subsection{Second alternative near supremum}
Here we will assume that either
\begin{equation} \label{e.mu+_bounds_nearzero} 
-\tfrac12 \boldsymbol{\omega} \leq \boldsymbol{\mu}^- \leq \tfrac14 \boldsymbol{\omega} \quad \text{ and }\quad \theta = \boldsymbol{\omega}^{q-1}
\end{equation}
holds true, or 
\begin{equation} \label{e.mu+_bounds_awayzero} 
\boldsymbol{\mu}^- > \tfrac14 \boldsymbol{\omega} \quad \text{ and }\quad \tfrac12 \boldsymbol{\mu}^+ \leq \theta^\frac{1}{q-1} \leq 5 \boldsymbol{\mu}^+.
\end{equation}
Note that \eqref{e.mu+_bounds_nearzero}$_1$ is equivalent to
\begin{equation*} 
\tfrac12 \boldsymbol{\omega} \leq \boldsymbol{\mu}^+ \leq \tfrac54 \boldsymbol{\omega}.
\end{equation*}

\begin{lemma} \label{l.de_giorgi_2_alternative}
Let $Q_{2\rho,\theta (2\rho)^2}(z_o) \Subset \Omega_T$ be a parabolic cylinder such that 
$$
\sup_{Q_{\rho, \theta \rho^2}(z_o)} \psi \leq \tfrac12 ( \boldsymbol{\mu}^+ + \boldsymbol{\mu}^- )
$$
and assume that hypothesis \eqref{e.mu+_bounds_nearzero} or \eqref{e.mu+_bounds_awayzero} holds.
For any $\nu \in (0,1)$ there exists constant $a = a(n,q,C_o,C_1,\nu) \in \big(0,\frac{1}{64}\big]$ such that if
$$
\left| \{ u - \boldsymbol{\mu}^- \leq \tfrac12 \boldsymbol{\omega} \} \cap Q_{\rho,\theta \rho^2}(z_o) \right| > \nu \left| Q_{\rho,\theta \rho^2}(z_o) \right|,
$$
then 
$$
\boldsymbol{\mu}^+ - u \geq a \boldsymbol{\omega}\quad \text{ a.e. in } Q_{\frac{\rho}{2},\frac12 \nu \theta \left( \frac{\rho}{2} \right)^2}(z_o).
$$
\end{lemma}

\section{Reduction of oscillation} \label{s.red-osc}
Throughout the rest of the paper, we denote the minimum of the parameters $\nu$ from Lemmas~\ref{l.de_giorgi_type+++} to~\ref{l.de_giorgi_awayzero-} by $\nu_o$.
Further, we let $a$ be the minimum of the respective parameters in Lemmas~\ref{l.de_giorgi_2_alternative2} and~\ref{l.de_giorgi_2_alternative} corresponding to the parameter $\nu_o$ chosen above, and define $\delta = 1-a \in (0,1)$.
This implies that these parameters coincide in the following lemmas, which allows us to use them subsequently in Section~\ref{s.final_argument}.

Moreover, throughout this section, we consider parabolic cylinders of the form $Q_o := Q_{\rho_o, \theta \rho_o^2}(z_o)$ and quantities
\begin{equation}
	\boldsymbol{\mu}^+_o \geq \esssup_{Q_o} u,
	\quad
	\boldsymbol{\mu}^-_o \leq \essinf_{Q_o} u
	\quad \text{ and }\quad
	\boldsymbol{\omega}_o = \boldsymbol{\mu}^+_o - \boldsymbol{\mu}^-_o.
	\label{eq:assumptions_quantities}
\end{equation}
Further, we assume that
\begin{equation}
	\sup_{Q_o} \psi \leq \tfrac12 (\boldsymbol{\mu}^+_o + \boldsymbol{\mu}^-_o)
	\quad \text{and} \quad
	\osc_{Q_o} \psi \leq \tfrac12 \boldsymbol{\omega}_o.
	\label{eq:assumptions_obstacle}
\end{equation}
We treat the following cases separately: Either we assume that
\begin{equation}
	\boldsymbol{\mu}^+_o \geq - \tfrac14 \boldsymbol{\omega}_o
	\quad \text{and}\quad
	\boldsymbol{\mu}^-_o \leq \tfrac14 \boldsymbol{\omega}_o
	\quad \text{and}\quad
	\theta = \boldsymbol{\omega}_o^{q-1},
	\label{eq:near_zero_reduction}
\end{equation}
or
\begin{equation}
	\boldsymbol{\mu}^-_o > \tfrac14 \boldsymbol{\omega}_o
	\quad \text{and}\quad
	\theta = (\boldsymbol{\mu}^+_o)^{q-1},
	\label{eq:above_zero_reduction}
\end{equation}
or
\begin{equation}
	\boldsymbol{\mu}^+_o < - \tfrac14 \boldsymbol{\omega}_o
	\quad \text{and}\quad
	\theta = \left| \boldsymbol{\mu}^-_o \right|^{q-1}.
	\label{eq:below_zero_reduction}
\end{equation}

First, we are concerned with the case where $u$ is near zero.

\begin{lemma} \label{l.red_oscillation_nearzero}
Assume that the hypotheses \eqref{eq:assumptions_quantities}, \eqref{eq:assumptions_obstacle} and \eqref{eq:near_zero_reduction} are satisfied.  
Define
$$
	\boldsymbol{\omega}_1 :=
	\max \Big\{ \delta \boldsymbol{\omega}_o, 2 \osc_{Q_o} \psi  \Big\}
$$
and
$$
	Q_1 := Q_{\rho_1,\theta_1 \rho_1^2}(z_o)
	\quad \text{ with }\quad
	\theta_1 = \boldsymbol{\omega}_1^{q-1},
	\ \rho_1 = \lambda \rho_o,\
	\lambda:= \sqrt{\frac{\nu_o}{8} \delta^{(1-q)_+}}.
$$
Then we have that
$$
\essosc_{Q_1} u \leq \boldsymbol{\omega}_1\quad \text{ and }\quad Q_1 \subset Q_o.
$$
\end{lemma}

\begin{proof}
Observe that \eqref{eq:near_zero_reduction} implies that $\left| \boldsymbol{\mu}^\pm_o \right| \leq \tfrac54 \boldsymbol{\omega}_o$. Furthermore, we must have
\begin{equation} \label{e.mu-alternatives-nearzero}
\boldsymbol{\mu}^+_o \geq \tfrac12 \boldsymbol{\omega}_o
\quad \text{ or }\quad
\boldsymbol{\mu}^-_o \leq - \tfrac12 \boldsymbol{\omega}_o,
\end{equation}
since otherwise we would end up in a contradiction. Suppose first that~\eqref{e.mu-alternatives-nearzero}$_1$ holds true. Then, we have $\tfrac12 \boldsymbol{\omega}_o \leq \boldsymbol{\mu}^+_o \leq \tfrac54 \boldsymbol{\omega}_o$. In this case we distinguish between the alternatives
\begin{equation}
	\label{e.alternatives_-}
	\left\{
	\begin{array}{l}
	\left| \left\{ u \leq \boldsymbol{\mu}^-_o + \tfrac12 \boldsymbol{\omega}_o \right\} \cap Q_o \right|
	\leq
	\nu_o \left| Q \right|, \\[5pt]
	\left| \left\{ u \leq \boldsymbol{\mu}^-_o + \tfrac12 \boldsymbol{\omega}_o \right\} \cap Q_o \right|
	>
	\nu_o \left| Q \right|.
	\end{array}
	\right.
\end{equation}
When~\eqref{e.alternatives_-}$_1$ holds true, we apply Lemma~\ref{l.de_giorgi_type---}. Since $|\boldsymbol{\mu}^-_o|\leq \tfrac54 \boldsymbol{\omega}_o$, this yields
$$
\essinf_{Q_{\frac{\rho_o}{2}, \theta \left( \frac{\rho_o}{2} \right)^2}(z_o)} u \geq \boldsymbol{\mu}^-_o + \tfrac14 \boldsymbol{\omega}_o.
$$
On the other hand, if~\eqref{e.alternatives_-}$_2$ holds true, we may apply Lemma~\ref{l.de_giorgi_2_alternative} and obtain
$$
\esssup_{Q_{\frac{\rho_o}{2}, \frac12\nu_o \theta \left( \frac{\rho_o}{2}\right)^2}(z_o)} u \leq \boldsymbol{\mu}^+_o - a \boldsymbol{\omega}_o.
$$
Clearly $\lambda \leq \tfrac12$ and in the case $0<q<1$ we can estimate
$$
\theta_1 \rho_1^2 = \boldsymbol{\omega}_1^{q-1} \cdot \tfrac12 \nu_o \delta^{1-q} \left( \tfrac{\rho_o}{2}\right)^2 \leq \tfrac12 \nu_o \theta \left( \tfrac{\rho_o}{2}\right)^2
$$
by $\boldsymbol{\omega}_1 \geq \delta \boldsymbol{\omega}_o$. If $q>1$ and $\boldsymbol{\omega}_1 = \delta \boldsymbol{\omega}_o$ the same inequality holds true.
In the case where $\boldsymbol{\omega}_1 = 2 \osc_{Q_o} \psi$, we use that $2 \osc_{Q_o} \psi \leq \boldsymbol{\omega}_o$ by assumption. Hence,
$$
Q_1 \subset Q_{\frac{\rho_o}{2}, \frac12\nu_o \theta \left( \frac{\rho_o}{2}\right)^2} \subset Q_{\frac{\rho_o}{2}, \theta \left( \frac{\rho_o}{2} \right)^2}
$$
follows in any case.
Therefore, we have that either 
$$
\essosc_{Q_1} u \leq \essosc_{Q_{\frac{\rho_o}{2}, \theta \left( \frac{\rho_o}{2} \right)^2}} u \leq \boldsymbol{\mu}^+_o - ( \boldsymbol{\mu}^-_o + \tfrac14 \boldsymbol{\omega}_o) = \tfrac34 \boldsymbol{\omega}_o \leq \delta \boldsymbol{\omega}_o \leq \boldsymbol{\omega}_1,
$$
or 
$$
\essosc_{Q_1} u \leq \essosc_{Q_{\frac{\rho_o}{2}, \frac12\nu_o \theta \left( \frac{\rho_o}{2}\right)^2}} u \leq \boldsymbol{\mu}^+_o - a\boldsymbol{\omega}_o - \boldsymbol{\mu}^-_o = \delta \boldsymbol{\omega}_o \leq \boldsymbol{\omega}_1.
$$
This completes the proof in case~\eqref{e.mu-alternatives-nearzero}$_1$.
Now, suppose that~\eqref{e.mu-alternatives-nearzero}$_2$ holds true.
Then, we have that $- \tfrac54 \boldsymbol{\omega}_o \leq \boldsymbol{\mu}^-_o \leq - \tfrac12 \boldsymbol{\omega}_o$ and distinguish between the alternatives
\begin{equation}
	\label{e.alternatives_+}
	\left\{
	\begin{array}{l}
	\left| \left\{ u \geq \boldsymbol{\mu}^+_o - \tfrac12 \boldsymbol{\omega}_o \right\}  \cap Q_o\right|
	\leq
	\nu_o \left| Q \right|, \\[5pt]
	\left| \left\{ u \geq \boldsymbol{\mu}^+_o - \tfrac12 \boldsymbol{\omega}_o \right\} \cap Q_o\right|
	>
	\nu_o \left| Q \right|.
	\end{array}
	\right.
\end{equation}
When~\eqref{e.alternatives_+} holds true, we may apply Lemma~\ref{l.de_giorgi_type+++}, since $|\boldsymbol{\mu}^+_o|\leq \tfrac 54\boldsymbol{\omega}_o$. This implies
$$
\esssup_{Q_{\frac{\rho_o}{2}, \theta \left( \frac{\rho_o}{2} \right)^2}(z_o)} u \leq \boldsymbol{\mu}^+_o - \tfrac14 \boldsymbol{\omega}_o.
$$
On the other hand if~\eqref{e.alternatives_+}$_2$ holds true, we apply Lemma~\ref{l.de_giorgi_2_alternative2} to obtain
$$
\essinf_{Q_{\frac{\rho_o}{2}, \frac12\nu_o \theta \left( \frac{\rho_o}{2}\right)^2}(z_o)} u \geq \boldsymbol{\mu}^-_o + a \boldsymbol{\omega}_o.
$$
By using similar estimates as above, we conclude that
$$
\essosc_{Q_1} u \leq \delta \boldsymbol{\omega}_o \leq \boldsymbol{\omega}_1,
$$
which finishes the proof.
\end{proof}

Up next we will prove a similar result in the case where $u$ is bounded away from zero and positive.

\begin{lemma} \label{l.reduction_of_osc_az+}
Assume that \eqref{eq:assumptions_quantities}, \eqref{eq:assumptions_obstacle} and \eqref{eq:above_zero_reduction} hold true.
For the sequence of cylinders
$$
Q_i := Q_{\rho_i, \theta \rho_i^2}(z_o),\quad \text{ with }\quad \rho_i = \lambda^i \rho_o,\quad \lambda := \sqrt{\frac{\nu_o}{8}}
$$
we define
$$ 
\boldsymbol{\omega}_i := \max \Big\{ \delta \boldsymbol{\omega}_{i-1}, 2 \osc_{Q_{i-1}} \psi \Big\}\quad \text{ for }i \in \N_0.
$$
Then, for any $i \in \N_0$ there holds
$$
\essosc_{Q_i} u \leq \boldsymbol{\omega}_i.
$$
\end{lemma}

\begin{proof}
First, observe that 
$$
Q_{i+1} \subset Q_{\frac{\rho_i}{2}, \frac12 \nu_o \theta \left( \frac{\rho_i}{2} \right)^2}(z_o) \subset Q_{\frac{\rho_i}{2}, \theta \left( \frac{\rho_i}{2} \right)^2} (z_o) \subset Q_i
$$
for any $i \in \N_0$. Define 
$$
	\boldsymbol{\mu}^-_i := \essinf_{Q_i} u,
	\quad \boldsymbol{\mu}^+_i = \boldsymbol{\mu}^-_i + \boldsymbol{\omega}_i
$$
for every $i \in \N$.
Now by the assumptions, we already have that $\boldsymbol{\omega}_1 \leq \boldsymbol{\omega}_o$ holds true.
By induction it follows directly that 
$$
	\tfrac14 \boldsymbol{\omega}_{i+1}
	\leq
	\tfrac14 \boldsymbol{\omega}_i
	<
	\boldsymbol{\mu}^-_i
	\leq
	\boldsymbol{\mu}^-_{i+1}
$$
for all $i \in \N_0$.
Since \eqref{eq:above_zero_reduction} is equivalent to $\boldsymbol{\mu}^+_o <  5 \boldsymbol{\mu}^-_o$, we have that $\theta^\frac{1}{q-1} = \boldsymbol{\mu}^+_o <  5 \boldsymbol{\mu}^-_o \leq  5 \boldsymbol{\mu}^-_i \leq 5 \boldsymbol{\mu}^+_i $.
Further, we know that $\boldsymbol{\mu}^+_i = \boldsymbol{\mu}^-_i + \boldsymbol{\omega}_i \leq \esssup_{Q_o} u + \boldsymbol{\omega}_o \leq 2 \boldsymbol{\mu}^+_o = 2 \theta^\frac{1}{q-1}$.
Therefore we find that
$$
	\min \big\{ \left( 5 \boldsymbol{\mu}^+_i \right)^{q-1},
	\left( \tfrac12 \boldsymbol{\mu}^+_i \right)^{q-1} \big\}
	\leq
	\theta
	\leq
	\max \big\{ \left( 5 \boldsymbol{\mu}^+_i \right)^{q-1},
	\left( \tfrac12 \boldsymbol{\mu}^+_i \right)^{q-1} \big\}
$$
for any $i \in \N_0$.
Moreover, we have that 
$$
\sup_{Q_i} \psi = \inf_{Q_i}\psi + \osc_{Q_i} \psi \leq \essinf_{Q_i} u + \osc_{Q_{i-1}} \psi \leq \boldsymbol{\mu}^-_i + \tfrac12 \boldsymbol{\omega}_i = \tfrac12 (\boldsymbol{\mu}^+_i + \boldsymbol{\mu}^-_i) 
$$
for every $i \in \N_0$.
Assume that $\essosc_{Q_i} u \leq \boldsymbol{\omega}_i$ for some $i \in \N$.
For $i=0$, this clearly holds by \eqref{eq:assumptions_quantities}.
Then, in particular we have that $\boldsymbol{\mu}^+_i = \boldsymbol{\mu}^-_i + \boldsymbol{\omega}_i \geq \essinf_{Q_i} u + \essosc_{Q_i} u = \esssup_{Q_i} u$.
Now we distinguish between the alternatives 
\begin{equation*}
	\left\{
	\begin{array}{l}
	\left| \left\{ u \leq \boldsymbol{\mu}^-_i + \tfrac12 \boldsymbol{\omega}_i \right\} \cap Q_i \right|
	\leq
	\nu_o \left| Q_i \right|, \\[5pt]
	\left| \left\{ u \leq \boldsymbol{\mu}^-_i + \tfrac12 \boldsymbol{\omega}_i \right\} \cap Q_i \right|
	>
	\nu_o \left| Q_i \right|.
	\end{array}
	\right.
\end{equation*}
When the first alternative holds true, by Lemma~\ref{l.de_giorgi_awayzero-} we obtain that
$$
\essinf_{Q_{\frac{\rho_i}{2}, \theta \left( \frac{\rho_i}{2} \right)^2}} u \geq \boldsymbol{\mu}^-_i + \tfrac14 \boldsymbol{\omega}_i.
$$
On the other hand if the second alternative holds true, Lemma~\ref{l.de_giorgi_2_alternative} implies that
$$
\esssup_{Q_{\frac{\rho_i}{2}, \frac12 \nu_o \theta \left( \frac{\rho_i}{2} \right)^2}} u \leq \boldsymbol{\mu}^+_i - a \boldsymbol{\omega}_i
$$
for some $a = a(n,q,C_o,C_1) \in (0,\tfrac{1}{64}]$.
Recalling that $\delta = 1-a$, we see that in both cases
$$
\essosc_{Q_{i+1}} u \leq \delta \boldsymbol{\omega}_i \leq \boldsymbol{\omega}_{i+1},
$$
which completes the proof.
\end{proof}

Finally we state and prove a similar lemma in the case where $u$ is bounded away from zero and negative.

\begin{lemma} \label{l.reduction_of_osc_az-}
Assume that the hypotheses \eqref{eq:assumptions_quantities}, \eqref{eq:assumptions_obstacle} and \eqref{eq:below_zero_reduction} hold.
For the sequence of cylinders
$$
Q_i := Q_{\rho_i, \theta \rho_i^2}(z_o),\quad \text{ with }\quad \rho_i = \lambda^i \rho_o,\quad \lambda := \sqrt{\frac{\nu_o}{8}}
$$
we define
$$
	\boldsymbol{\omega}_i := \max \Big\{ \delta \boldsymbol{\omega}_{i-1}, 2 \osc_{Q_{i-1}} \psi \Big\}\quad \text{ for }i \in \N_0.
$$
Then, for any $i \in \N_0$ there holds
$$
\essosc_{Q_i} u \leq \boldsymbol{\omega}_i.
$$
\end{lemma}

\begin{proof}
First, observe that 
$$
Q_{i+1} \subset Q_{\frac{\rho_i}{2}, \frac12 \nu_o \theta \left( \frac{\rho_i}{2} \right)^2}(z_o) \subset Q_{\frac{\rho_i}{2}, \theta \left( \frac{\rho_i}{2} \right)^2} (z_o) \subset Q_i
$$
for any $i \in \N_0$. Define 
$$
\boldsymbol{\mu}^-_i := \essinf_{Q_i} u, \quad \boldsymbol{\mu}^+_i = \min \left\{ \boldsymbol{\mu}^-_i + \boldsymbol{\omega}_i, \boldsymbol{\mu}_{i-1}^+ \right \}
$$
for every $i \in \N$. Now by the assumptions, we already have that $\boldsymbol{\omega}_1 \leq \boldsymbol{\omega}_o$ holds true and by induction it directly follows that
$$
	\boldsymbol{\omega}_{i+1} \leq \boldsymbol{\omega}_i
	\quad \text{and} \quad
	\boldsymbol{\mu}^+_{i+1}
	\leq
	\boldsymbol{\mu}^+_i
	<
	- \tfrac14 \boldsymbol{\omega}_i
	\leq
	- \tfrac14 \boldsymbol{\omega}_{i+1}
$$
for all $i \in \N_0$, where we have used that $\{\boldsymbol{\mu}^+_i\}_{i\in\N_0}$ is a nonincreasing sequence by definition.
Since \eqref{eq:below_zero_reduction} is equivalent to $\boldsymbol{\mu}^-_o > 5 \boldsymbol{\mu}^+_o$ and $\{\boldsymbol{\mu}^+_i\}_{i\in\N_0}$ is nonincreasing, we have that $\theta^\frac{1}{q-1} = - \boldsymbol{\mu}^-_o <  -5 \boldsymbol{\mu}^+_o \leq -5 \boldsymbol{\mu}^-_i $
and that $\theta^\frac{1}{q-1} = -\boldsymbol{\mu}^-_o \geq - \boldsymbol{\mu}^-_i$; i.e., we find that
$$
	\min \big\{ \left| 5 \boldsymbol{\mu}^-_i \right|^{q-1},
	\left| \boldsymbol{\mu}^-_i \right|^{q-1} \big\}
	\leq
	\theta
	\leq
	\max \big\{ \left| 5 \boldsymbol{\mu}^-_i \right|^{q-1},
	\left| \boldsymbol{\mu}^-_i \right|^{q-1} \big\}
$$ 
for any $i \in \N_0$.
Up next we show that the condition $\sup_{Q_i}\psi \leq \tfrac12 \left( \boldsymbol{\mu}^+_i + \boldsymbol{\mu}^-_i \right)$ holds true for all $i\in \N_0$. For $i=0$ this is part of hypothesis \eqref{eq:assumptions_obstacle}. Suppose that this holds for some $i \in \N$.
On the one hand if $\boldsymbol{\mu}^+_{i+1} = \boldsymbol{\mu}^-_{i+1} + \boldsymbol{\omega}_{i+1}$, we have that
$$
\sup_{Q_{i+1}} \psi = \inf_{Q_{i+1}}\psi + \osc_{Q_{i+1}} \psi \leq \essinf_{Q_{i+1}} u + \osc_{Q_{i}} \psi \leq \boldsymbol{\mu}^-_{i+1} + \tfrac12 \boldsymbol{\omega}_{i+1} = \tfrac12 (\boldsymbol{\mu}^+_{i+1} + \boldsymbol{\mu}^-_{i+1}). 
$$
On the other hand if $\boldsymbol{\mu}^+_{i+1} = \boldsymbol{\mu}_{i}^+$,
by the induction assumption and the property that $\boldsymbol{\mu}^-_i \leq \boldsymbol{\mu}^-_{i+1}$ we may estimate
$$
\sup_{Q_{i+1}}\psi \leq \sup_{Q_{i}} \psi \leq \tfrac12 \left( \boldsymbol{\mu}^+_i + \boldsymbol{\mu}^-_i \right) \leq \tfrac12 \left( \boldsymbol{\mu}^+_{i+1} + \boldsymbol{\mu}^-_{i+1} \right).
$$
Next we want to show that for every $i \in \N$, there holds
\begin{equation} \label{e.az_induction-}
\esssup_{Q_{i-1}}u \leq \boldsymbol{\mu}^+_{i-1} \quad \text{ and }\quad \essosc_{Q_i} u \leq \boldsymbol{\omega}_i.
\end{equation}
For $i =1$,~\eqref{e.az_induction-}$_1$ clearly holds by assumption.
Now, we consider the alternatives
\begin{equation*}
	\left\{
	\begin{array}{l}
	\left| \left\{ u \geq \boldsymbol{\mu}^+_o - \tfrac12 \boldsymbol{\omega}_o \right\} \cap Q_o \right|
	\leq
	\nu_o \left| Q_o \right|, \\[5pt]
	\left| \left\{ u \geq \boldsymbol{\mu}^+_o - \tfrac12 \boldsymbol{\omega}_o \right\} \cap Q_o \right|
	>
	\nu_o \left| Q_o \right|.
	\end{array}
	\right.
\end{equation*}
If the first alternative holds true, we may apply Lemma~\ref{l.de_giorgi_awayzero+}. On the other hand, if the second alternative holds true, we use Lemma~\ref{l.de_giorgi_2_alternative2}.
In both cases, we find that 
$$
\essosc_{Q_1} u \leq \delta \boldsymbol{\omega}_o \leq \boldsymbol{\omega}_1,
$$
where the last inequality holds by definition of $\boldsymbol{\omega}_1$, and $\delta = 1-a$ with the constant $a$ from Lemma~\ref{l.de_giorgi_2_alternative2}. This takes care of the case $i=1$. Now let us suppose that~\eqref{e.az_induction-} holds for some $i\in \N$. It follows that either we have
$$
\boldsymbol{\mu}^+_{i} = \boldsymbol{\mu}^-_{i} + \boldsymbol{\omega}_{i} \geq \essinf_{Q_i} u + \essosc_{Q_i} u = \esssup_{Q_i} u
$$
with assumption~\eqref{e.az_induction-}$_2$ or that 
$$
\boldsymbol{\mu}^+_i = \boldsymbol{\mu}^+_{i-1} \geq \esssup_{Q_{i-1}} u \geq \esssup_{Q_i} u
$$
with assumption~\eqref{e.az_induction-}$_1$. The two inequalities above already prove~\eqref{e.az_induction-}$_1$. Let us define $ \tilde{\boldsymbol{\omega}}_i = \boldsymbol{\mu}^+_i - \boldsymbol{\mu}^-_i \leq \boldsymbol{\omega}_i$. Now we will use the alternatives
\begin{equation}
	\label{e.alternatives_+_awayzero-}
	\left\{
	\begin{array}{l}
	\left| \left\{ u \geq \boldsymbol{\mu}^+_i - \tfrac12 \tilde{\boldsymbol{\omega}}_i \right\} \cap Q_i \right|
	\leq
	\nu_o \left| Q_i \right|, \\[5pt]
	\left| \left\{ u \geq \boldsymbol{\mu}^+_i - \tfrac12 \tilde{\boldsymbol{\omega}}_i \right\} \cap Q_i \right|
	>
	\nu_o \left| Q_i \right|.
	\end{array}
	\right.
\end{equation}
In the first case, we apply Lemma~\ref{l.de_giorgi_awayzero+} with $ \tilde{\boldsymbol{\omega}}_i$ in place of $\boldsymbol{\omega}_i$. This implies that 
$$
\esssup_{Q_{\frac{\rho_i}{2}, \theta \left( \frac{\rho_i}{2} \right)^2}} u \leq \boldsymbol{\mu}^+_i - \tfrac14 \tilde{\boldsymbol{\omega}}_i.
$$
If~\eqref{e.alternatives_+_awayzero-}$_2$ holds true, we use Lemma~\ref{l.de_giorgi_2_alternative2}, which gives us
$$
\essinf_{Q_{\frac{\rho_i}{2}, \frac12 \nu_o \theta \left( \frac{\rho_i}{2} \right)^2}} u \geq \boldsymbol{\mu}^-_i + a \tilde{\boldsymbol{\omega}}_i
$$
for some $a = a(n,q,C_o,C_1) \in (0,\tfrac{1}{64}]$.
Recalling that $\delta = 1-a$, the preceding two inequalities both imply~\eqref{e.az_induction-}$_2$, which completes the proof.
\end{proof}

\section{The final argument} \label{s.final_argument}
\subsection{The final argument begins}
In the following we assume that $u$ is globally bounded for ease of notation. However, the argument holds for a locally bounded weak solution $u$ by restricting to a compact subset of $\Omega_T$. Thus, we can assume that 
\begin{equation} \label{e.rescaling}
\essosc_{\Omega_T} u \leq 1,\quad \big|\esssup_{\Omega_T} u\big|\leq \tfrac12\quad  \text{and }\quad \big|\essinf_{\Omega_T}  u\big| \leq \tfrac12
\end{equation}
by using a rescaling argument as in Appendix~\ref{appendix_a} with $M = 2 \|u\|_\infty$. 

Assume that $\psi \in C^{0; \beta, \frac{\beta}{2}}(\Omega_T)$ for the exponent $\beta \in (0,1)$, i.e.
$$
	[\psi]_{0; \beta, \frac{\beta}{2}}
	:=
	\sup_{(x,t), (y,s) \in \Omega_T}
	\frac{|\psi(x,t) - \psi(y,s)|}{\max\{ |x-y|^\beta, |t-s|^\frac{\beta}{2}\}}
	< \infty.
$$
Let $\epsilon = \frac{2 \beta (1-q)_+}{2 + \beta (1-q)_+} \in [0,2)$
and $\gamma_o = \frac{2 \beta}{2 + \beta (1-q)_+} = \big( 1 - \frac{\epsilon}{2} \big) \beta \in (0, \beta]$. Observe that $\eps = 0$ and $\gamma_o = \beta$ in the singular case $q>1$.
Further, consider an arbitrary point $z_o = (x_o, t_o) \in \Omega_T$ and let $R \in (0,1)$ be so small that $Q_{R, R^{2-\epsilon}}(z_o) \Subset \Omega_T$.
In the following, we omit $z_o$ from our notation for simplicity.
Next, we consider the quantity
$$
	\Psi(\rho)
	:=
	\max \big\{ \rho^{\gamma_o}, 2 \osc_{Q_{\rho, \rho^{2-\epsilon}}} \psi \big\}.
$$
By the assumption $\psi \in C^{0; \beta, \frac{\beta}{2}}(\Omega_T)$, the choice of $\epsilon$ and the fact that $\rho \in (0,1)$ we obtain that
$$
	\osc_{Q_{\rho, \rho^{2-\epsilon}}} \psi
	\leq
	[\psi]_{0; \beta, \frac{\beta}{2}}
	\max \big\{ \rho^\beta, \rho^{(2-\epsilon) \frac{\beta}{2}} \big\}
	=
	[\psi]_{0; \beta, \frac{\beta}{2}} \rho^{\gamma_o}.
$$
Thus, we conclude that $u$ is H\"older continuous at $(x_o, t_o)$ in the case that the bound
$$
	\essosc_{Q_{\rho, \rho^{2-\epsilon}}} u
	\leq
	\Psi(\rho)
	\quad \forall \rho \in (0,R]
$$
holds.
In order to prove the H\"older continuity of $u$ in the opposite case,
we set $\rho_o = R$ if
$$
	\Psi(R)
	<
	\essosc_{Q_{R, R^{2-\epsilon}}} u.
$$
Otherwise, there exists $\rho_o \in (0,R)$ such that
\begin{equation}
	\label{e.rho_o}
	\Psi(\rho_o) < \essosc_{Q_{\rho_o, \rho_o^{2-\epsilon}}} u
	\qquad \text{and} \qquad
	\essosc_{Q_{\rho, \rho^{2-\epsilon}}} u \leq 2 \Psi(\rho)
	\quad \forall \rho \in [\rho_o, R].
\end{equation}
For this choice of $\rho_o$, we define
$$
	\boldsymbol{\mu}^+_o := \esssup_{Q_{\rho_o, \rho_o^{2-\epsilon}}} u,
	\quad
	\boldsymbol{\mu}^-_o := \essinf_{Q_{\rho_o, \rho_o^{2-\epsilon}}} u,
	\quad
	\boldsymbol{\omega}_o := \boldsymbol{\mu}^+_o - \boldsymbol{\mu}^-_o.
$$
In the case $0<q<1$ we define $\theta_o := \boldsymbol{\omega}_o^{q-1}$.
Further, we compute that
$$
	\theta_o
	=
	\Big( \essosc_{Q_{\rho_o, \rho_o^{2-\epsilon}}} u \Big)^{q-1}
	<
	\Psi(\rho_o)^{q-1}
	\leq
	\rho_o^{-\epsilon}
$$
by definition of $\rho_o$.
In the singular case $q>1$, we define
$$
	\theta_o :=
	\left\{
	\begin{array}{ll}
		\boldsymbol{\omega}^{q-1}_o
		&\text{if $\boldsymbol{\mu}^+_o \geq -\frac{1}{4} \boldsymbol{\omega}_o$ and
		$\boldsymbol{\mu}^-_o \leq \frac{1}{4} \boldsymbol{\omega}_o$}, \\
		\left( \boldsymbol{\mu}^+_o \right)^{q-1}
		&\text{if $\boldsymbol{\mu}^-_o > \frac{1}{4} \boldsymbol{\omega}_o$}, \\
		\left| \boldsymbol{\mu}^-_o\right|^{q-1}
		&\text{if $\boldsymbol{\mu}^+_o < -\frac{1}{4} \boldsymbol{\omega}_o$}. \\
	\end{array}
	\right.
$$
By taking into account~\eqref{e.rescaling} we conclude that in any case $\theta_o \leq 1$. Hence, we have the set inclusion
$$
	Q_o :=
	Q_{\rho_o, \theta_o \rho_o^2}
	\subset
	Q_{\rho_o, \rho_o^{2-\epsilon}}.
$$
Since $u \geq \psi$ a.e.~in $\Omega_T$ and by the choice of $\rho_o$, we deduce that
\begin{align}
	\sup_{Q_o} \psi
	&\leq
	\sup_{Q_{\rho_o, \rho_o^{2-\epsilon}}} \psi
	=
	\inf_{Q_{\rho_o, \rho_o^{2-\epsilon}}} \psi
	+ \osc_{Q_{\rho_o, \rho_o^{2-\epsilon}}} \psi
	<
	\essinf_{Q_{\rho_o, \rho_o^{2-\epsilon}}} u
	+ \tfrac12 \essosc_{Q_{\rho_o, \rho_o^{2-\epsilon}}} u \nonumber\\
	&=
	\boldsymbol{\mu}^-_o + \tfrac12 \boldsymbol{\omega}_o
	=
	\tfrac12 \big( \boldsymbol{\mu}^+_o + \boldsymbol{\mu}^-_o \big).
	\label{e.psi}
\end{align}
By the definition of $\Psi$ and~\eqref{e.rho_o}$_1$ we infer
\begin{equation} \label{e.psi_2}
\osc_{Q_o} \psi \leq \osc_{Q_{\rho_o,\rho_o^{2-\eps}}}\psi \leq \tfrac12 \Psi(\rho_o) < \tfrac12 \boldsymbol{\omega}_o = \tfrac12 \left( \boldsymbol{\mu}_o^+ - \boldsymbol{\mu}_o^- \right).
\end{equation}
Now, we proceed as follows:
Suppose that $i_o \in \N_0 \cup \{\infty\}$ is the largest index for which we have
\begin{equation} \label{e.near_zero}
	\boldsymbol{\mu}^+_i \geq -\tfrac14 \boldsymbol{\omega}_i
	\quad \text{ and }\quad
	\boldsymbol{\mu}^-_i \leq \tfrac14 \boldsymbol{\omega}_i
\end{equation}
for all $i \in \{ 0,1, \ldots,i_o -1  \}$.
This means that up to the index $i_o - 1$, we apply the reduction of oscillation for the case where $u$ is near zero (see Section~\ref{s.iteration_nearzero}).
For every $i \geq i_o$ it then follows that either
\begin{equation} \label{e.away_zero}
\boldsymbol{\mu}^-_i > \tfrac14 \boldsymbol{\omega}_i\quad \text{ or }\quad \boldsymbol{\mu}^+_i < - \tfrac14 \boldsymbol{\omega}_i
\end{equation}
holds true and we use the results on reduction of oscillation for either one of the cases where $u$ is away from zero (see Sections~\ref{s.iteration_awayzero+}--\ref{s.iteration_awayzero-}).
Observe that it is also possible that $i_o = 0$, which means that the case where $u$ is near zero never occurs, or that $i_o = \infty$ and the iteration is carried out for $u$ near zero completely.

\subsection{Reduction of oscillation near zero} \label{s.iteration_nearzero}
Suppose that $i_o > 0$, otherwise we can skip this part and move directly to either Section~\ref{s.iteration_awayzero+} or~\ref{s.iteration_awayzero-} depending on which case in~\eqref{e.away_zero} holds true.
For $i \in \{ 1,2,\ldots ,i_o \}$, define 
\begin{equation*}
	\left\{
	\begin{array}{c}
		\rho_i := \lambda \rho_{i-1},
		\quad
		\boldsymbol{\omega}_i :=
		\max\{ \delta \boldsymbol{\omega}_{i-1}, 2 \osc_{Q_{i-1}} \psi \},
		\quad
		\theta_i := \boldsymbol{\omega}_i^{q-1}, \\[5pt]
		\quad
		\lambda := \sqrt{\tfrac18 \nu_o \delta^{(1-q)_+}}
		\quad
		Q_i := Q_{\rho_i,\theta_i \rho_i^2}, \\[5pt]
		\boldsymbol{\mu}_i^- := \essinf_{Q_i} u
		\quad \text{ and } \quad
		\boldsymbol{\mu}_i^+ := \boldsymbol{\mu}_i^- + \boldsymbol{\omega}_i.
	\end{array}
	\right.
\end{equation*}
This implies 
$$
\sup_{Q_i} \psi = \inf_{Q_i} \psi + \osc_{Q_i} \psi \leq \essinf_{Q_i} u + \osc_{Q_{i-1}} \psi \leq \boldsymbol{\mu}^-_i + \tfrac12 \boldsymbol{\omega}_i = \tfrac12 \left( \boldsymbol{\mu}^-_i + \boldsymbol{\mu}^+_i \right),
$$
and
$$
\osc_{Q_i} \psi \leq \osc_{Q_{i-1}} \psi \leq \tfrac12 \boldsymbol{\omega}_i = \tfrac12 \left( \boldsymbol{\mu}_i^+ - \boldsymbol{\mu}_i^- \right) 
$$
for every $i \in \{1,2,\ldots,i_o\}$. Now we claim that 
\begin{equation} \label{e.iteration_nearzero}
Q_{i_o} \subset Q_{i_o-1} \subset \cdots \subset Q_o\quad \text{ and }\quad \essosc_{Q_i} u \leq \boldsymbol{\omega}_i \quad \text{ for all } i\in \{0,1, \ldots , i_o\}.
\end{equation}
Clearly this holds true by definitions when $i=0$. Suppose that the statement holds true for some $i < i_o$. Then we have that
$$
\boldsymbol{\mu}^+_i = \boldsymbol{\mu}^-_i + \boldsymbol{\omega}_i \geq \boldsymbol{\mu}^-_i + \essosc_{Q_i} u =  \esssup_{Q_i} u,
$$
by assumption. Now we are in a point of using Lemma~\ref{l.red_oscillation_nearzero}, which implies that 
$$
Q_{i+1} \subset Q_i 
\quad
\text{ and }
\quad
\essosc_{Q_{i+1}} u \leq \boldsymbol{\omega}_{i+1},
$$
which proves~\eqref{e.iteration_nearzero}.

\subsection{Reduction of oscillation above and away from zero} \label{s.iteration_awayzero+}
Suppose that $i_o \in \N_0$ is the first index for which there holds that 
\begin{equation} \label{e.abovezero_io}
\boldsymbol{\mu}^-_{i_o} > \tfrac14 \boldsymbol{\omega}_{i_o}.
\end{equation}
Now we define $\theta_* = \left( \boldsymbol{\mu}^+_{i_o} \right)^{q-1}$.
In the case $0<q<1$ it directly follows that  $\theta_* \leq \left( \boldsymbol{\omega}_{i_o} \right)^{q-1} = \theta_{i_o}$.
If $q > 1$ and $i_o = 0$ we have $\theta_* = \theta_o$. If $i_o \geq 1$ we deduce that
$$
\tfrac14 \boldsymbol{\omega}_{i_o} < \boldsymbol{\mu}^-_{i_o} = \essinf_{Q_{i_o}} u \leq \esssup_{Q_{i_o-1}} u \leq \boldsymbol{\mu}^+_{i_o-1} \leq \tfrac54 \boldsymbol{\omega}_{i_o-1} \leq \tfrac{5}{4 \delta} \boldsymbol{\omega}_{i_o},
$$
by using~\eqref{e.near_zero} for the index $i_o-1$.
Since ~\eqref{e.abovezero_io} is equivalent to $\boldsymbol{\mu}^+_{i_o} < 5 \boldsymbol{\mu}^-_{i_o}$, this implies that
$$
\boldsymbol{\mu}^+_{i_o} < 5 \boldsymbol{\mu}^-_{i_o} \leq \tfrac{25}{4 \delta} \boldsymbol{\omega}_{i_o}.
$$
For the cylinders $i > i_o$ let 
$$
Q_i:= Q_{\hat \rho_i,\theta_* \hat \rho_i^2} \quad \text{ with } \hat \rho_i = \hat\lambda^{i-i_o} \left( \tfrac{4 \delta}{25} \right)^\frac{(q-1)_+}{2} \rho_{i_o},\ \hat\lambda := \sqrt{\frac{\nu_o}{8}},
$$
and let $Q^*_{i_o}= Q_{\hat \rho_{i_o},\theta_* \hat \rho_{i_o}^2} \subset Q_{i_o}$, where $Q_{i_o}$ is the cylinder obtained in the last section after application of the last iteration step or it is $Q_{i_o} = Q_o$ if $i_o = 0$.
Now we clearly have that $Q_{i_o} \supset Q_{i_o}^{*} \supset Q_{i_o+1} \supset \ldots$ and 
$$
\essinf_{Q^*_{i_o}} u \geq \essinf_{Q_{i_o}} u = \boldsymbol{\mu}^-_{i_o}\quad \text{ and }\quad \esssup_{Q^*_{i_o}} u \leq \esssup_{Q_{i_o}} u \leq \boldsymbol{\mu}^+_{i_o}.
$$
Further, we find that
$$
\sup_{Q^*_{i_o}} \psi \leq \sup_{Q_{i_o}} \psi \leq \tfrac12 \left( \boldsymbol{\mu}^+_{i_o} + \boldsymbol{\mu}^-_{i_o} \right),
$$
where the last inequality follows from Section~\ref{s.iteration_nearzero} if $i_o > 0$ and from~\eqref{e.psi} if $i_o = 0$.
Finally, we obtain that
$$
\boldsymbol{\mu}^+_{i_o} - \boldsymbol{\mu}^-_{i_o} = \boldsymbol{\omega}_{i_o} \geq 2 \osc_{Q_{i_o-1}} \psi \geq 2\osc_{Q_{i_o}} \psi \geq 2\osc_{Q^*_{i_o}} \psi,
$$
which follows from Section~\ref{s.iteration_nearzero} if $i_o>0$, and from~\eqref{e.psi_2} if $i_o = 0$. Now we are in the position to use Lemma~\ref{l.reduction_of_osc_az+}, which implies
$$
\essosc_{Q_{i}} u \leq \boldsymbol{\omega}_i \quad \text{ for all } i > i_o.
$$

\subsection{Reduction of oscillation below and away from zero} \label{s.iteration_awayzero-}
Suppose that $i_o \in \N_0$ is the first index for which there holds that 
\begin{equation} \label{e.belowzero_io}
\boldsymbol{\mu}^+_{i_o} < -\tfrac14 \boldsymbol{\omega}_{i_o}.
\end{equation}
Now we define $\theta_* = \left( - \boldsymbol{\mu}^-_{i_o} \right)^{q-1}$. If $0<q<1$, it follows that $\theta_* \leq \left( \boldsymbol{\omega}_{i_o} \right)^{q-1} = \theta_{i_o}$. In the case $q> 1$, $\theta_* = \theta_{i_o}$ if $i_o = 0$, and if $i_o > 0$ we deduce that
$$
- \tfrac14 \boldsymbol{\omega}_{i_o} > \boldsymbol{\mu}^+_{i_o} \geq \esssup_{Q_{i_o}} u \geq \essinf_{Q_{i_o-1}} u = \boldsymbol{\mu}^-_{i_o-1} \geq - \tfrac54 \boldsymbol{\omega}_{i_o-1} \geq - \tfrac{5}{4\delta} \boldsymbol{\omega}_{i_o}
$$
by using the condition~\eqref{e.near_zero} for the index $i_o-1$ in the penultimate inequality.
Since~\eqref{e.belowzero_io} is equivalent to $5 \boldsymbol{\mu}^+_{i_o} < \boldsymbol{\mu}^-_{i_o}$, this implies
$$
- \boldsymbol{\mu}^-_{i_o} < - 5 \boldsymbol{\mu}^+_{i_o} \leq \tfrac{25}{4 \delta} \boldsymbol{\omega}_{i_o}.
$$
For the cylinders $i > i_o$ we define 
$$
Q_i:= Q_{\hat \rho_i,\theta_* \hat \rho_i^2}, \quad \text{ with } \hat \rho_i = \hat\lambda^{i-i_o} \left( \tfrac{4 \delta }{25} \right)^\frac{(q-1)_+}{2} \rho_{i_o},\ \hat\lambda := \sqrt{\frac{\nu_o}{8}},
$$
and let $Q^*_{i_o}= Q_{\hat \rho_{i_o},\theta_* \hat \rho_{i_o}^2} \subset Q_{i_o}$, where $Q_{i_o}$ is the cylinder obtained in Section~\ref{s.iteration_nearzero} after the last iteration step. Observe that we have 
$$
\essinf_{Q^*_{i_o}} u \geq \essinf_{Q_{i_o}} u = \boldsymbol{\mu}^-_{i_o}\quad \text{ and }\quad \esssup_{Q^*_{i_o}} u \leq \esssup_{Q_{i_o}} u \leq \boldsymbol{\mu}^+_{i_o}.
$$
Further, we find that
$$
\sup_{Q^*_{i_o}} \psi \leq \sup_{Q_{i_o}} \psi \leq \tfrac12 \left( \boldsymbol{\mu}^+_{i_o} + \boldsymbol{\mu}^-_{i_o} \right),
$$
where the last inequality follows from Section~\ref{s.iteration_nearzero} if $i_o > 0$ and from~\eqref{e.psi} if $i_o = 0$.
Finally, from the results of Section~\ref{s.iteration_nearzero} when $i_o> 0$ and from~\eqref{e.psi_2} when $i_o = 0$ we conclude that
$$
\boldsymbol{\mu}^+_{i_o} - \boldsymbol{\mu}^-_{i_o} = \boldsymbol{\omega}_{i_o} \geq 2 \osc_{Q_{i_o-1}} \psi \geq 2\osc_{Q_{i_o}} \psi \geq 2\osc_{Q^*_{i_o}} \psi.
$$
Now we are in the position to use Lemma~\ref{l.reduction_of_osc_az-}, which implies
$$
\essosc_{Q_i} u \leq \boldsymbol{\omega}_i \quad \text{ for all } i > i_o.
$$

\subsection{Proof concluded}
We define
$$
	r_i :=
	\left\{
	\begin{array}{ll}
		\left( \tfrac{\delta^i}{4} \boldsymbol{\omega}_o \right)^\frac{q-1}{2} \hat \rho_i
		&\text{if } q > 1, \\
		\min \left\{1, \left( \tfrac{25}{4} \boldsymbol{\omega}_o \right)^\frac{q-1}{2} \right\} \hat \rho_i
		&\text{if } 0<q<1,
	\end{array}
	\right.
$$
where $\hat \rho_i = \rho_i$ for $i < i_o$.
We claim that
$$
Q_{r_i} := Q_{r_i, r_i^2} \subset Q_i
$$
for any $i \in \N_0$. 
Indeed, if $q>1$ we know that $\theta_i = \boldsymbol{\omega}_i^{q-1} \geq \left( \delta^i \boldsymbol{\omega}_o \right)^{q-1}$ for $i \leq i_o$ and $\theta_* \geq \left( \tfrac14 \boldsymbol{\omega}_{i_o} \right)^{q-1} \geq \left( \tfrac{\delta^{i_o}}{4} \boldsymbol{\omega}_o \right)^{q-1}$.
Moreover, if $0<q<1$ we have that $\theta_i = \boldsymbol{\omega}_i^{q-1} \geq \boldsymbol{\omega}_o^{q-1}$ for $i \leq i_o$ and $\theta_* \geq \left( \tfrac{25}{4} \boldsymbol{\omega}_{i_o-1} \right)^{q-1} \geq \left( \tfrac{25}{4} \boldsymbol{\omega}_o \right)^{q-1}$ by definition of $i_o$.
Therefore, we find that 
\begin{equation} \label{e.osc_u}
\essosc_{Q_{r_i}} u \leq \essosc_{Q_i} u \leq \boldsymbol{\omega}_i \leq \delta^i \boldsymbol{\omega}_o + 2 \sum_{j=0}^{i-1} \delta^j \osc_{Q_{i-1-j}} \psi.
\end{equation}
When $i-1-j\leq i_o$, by the fact that $\rho_{i-1-j} \leq \lambda^{i-1-j} \rho_o$ and the definition of $\lambda$ we estimate
\begin{align*}
\osc_{Q_{i-1-j}} \psi &\leq c \left( \rho_{i-1-j}^\beta + ( \theta_{i-1-j} \rho_{i-1-j}^2)^\frac{\beta}{2}  \right) = c\left( 1 + \boldsymbol{\omega}_{i-1-j}^\frac{\beta(q-1)}{2} \right) \rho_{i-1-j}^\beta \\
&\leq c \left( 1+ \delta^{- \frac{\beta(i-1-j)(1-q)_+}{2}} \boldsymbol{\omega}_o^\frac{\beta (q-1)}{2} \right) \rho_{i-1-j}^\beta \\
&\leq c \left( 1 + \boldsymbol{\omega}_o^\frac{\beta(q-1)}{2} \right) \delta^{-\frac{(i-1-j)\beta (1-q)_+}{2}} \rho_{i-1-j}^\beta \\
&\leq c (1+ \boldsymbol{\omega}_o^\frac{\beta(q-1)}{2} ) \left( \tfrac{\nu_o}{8} \right)^\frac{\beta (i-1-j)}{2} \rho_o^\beta.
\end{align*}
If $i-1-j > i_o$, by the definitions of $\lambda$ and $\hat \lambda$ we obtain in a similar way that
\begin{align*}
\osc_{Q_{i-1-j}} \psi &\leq c \left( \hat \rho_{i-1-j}^\beta + (\theta_* \hat \rho_{i-1-j}^2)^\frac{\beta}{2} \right) \leq c \left( 1 + \theta_*^\frac{\beta}{2} \right) \hat \rho_{i-1-j}^\beta \\
&\leq c \left(1 + \delta^{-\frac{i_o\beta (1-q)_+}{2}} \boldsymbol{\omega}_o^\frac{\beta (q-1)}{2} \right) \hat \rho_{i-1-j}^\beta \leq c \left( 1 + \boldsymbol{\omega}_o^\frac{\beta (q-1)}{2} \right) \delta^{- \frac{i_o \beta (1-q)_+}{2}} \hat \rho_{i-1-j}^\beta \\
&\leq c \big( 1 + \boldsymbol{\omega}_o^\frac{\beta (q-1)}{2} \big) \left( \tfrac{\nu_o}{8} \right)^\frac{\beta (i-1-j)}{2} \rho_o^\beta.
\end{align*}
Using the estimates above in~\eqref{e.osc_u} gives us
\begin{equation*}
\essosc_{Q_{r_i}} u \leq \delta^i \boldsymbol{\omega}_o + c (1 + \boldsymbol{\omega}_o^\frac{\beta (q-1)}{2}) \rho_o^\beta \sum_{j=0}^{i-1} \delta^j \left( \tfrac{\nu_o}{8} \right)^\frac{\beta (i-1-j)}{2}.
\end{equation*}
Setting 
$$
\tau := \max \left\{ \delta, \left( \tfrac{\nu_o}{8} \right)^\frac{\beta}{2} \right\},
$$
we conclude from the preceding inequality that
\begin{equation*}
\essosc_{Q_{r_i}} u \leq \delta^i \boldsymbol{\omega}_o + c i \tau^{i-1} (1 + \boldsymbol{\omega}_o^\frac{\beta (q-1)}{2}) \rho_o^\beta.
\end{equation*}
By the fact that $i \sqrt{\tau}^i \leq -2/(e \log \tau)$, we infer
\begin{equation*}
\essosc_{Q_{r_i}} u \leq \delta^i \boldsymbol{\omega}_o + c \sqrt{\tau}^{i} (1 + \boldsymbol{\omega}_o^\frac{\beta (q-1)}{2}) \rho_o^\beta
\end{equation*}
for a constant $c = c(n,q,C_o,C_1,\beta, [\psi]_{0;\beta,\frac{\beta}{2}})$. Let us define $\eta = \left( \tfrac{\delta}{5} \right)^{(q-1)_+} \lambda $ and
$$
\gamma_1 := \min \left\{ \frac{\log \tau}{2 \log \eta},  \frac{2\beta}{2+\beta |q-1|} \right\},
$$
and observe that $\gamma_1 \leq \frac{\log \delta}{\log \eta}$.
By using the fact that $\eta^i \boldsymbol{\omega}_o^\frac{q-1}{2} \leq \frac{r_i}{\rho_o}$,~\eqref{e.rho_o} and $\boldsymbol{\omega}_o \leq 1$ in the case $q> 1$, and $\boldsymbol{\omega}_o^{q-1} \leq \rho_o^{-\eps}$,~\eqref{e.rho_o}, $\eta^i \leq \frac{\hat \rho_i}{\rho_o}$ and $\hat \rho_i \leq c(q,\| u \|_\infty) r_i$ when $0<q<1$, we may estimate
\begin{align} \label{e.osc_u_ri}
\essosc_{Q_{r_i}} u &\leq \eta^{\gamma_1 i} \boldsymbol{\omega}_o + c \eta^{\gamma_1 i} (1  + \boldsymbol{\omega}_o^\frac{\beta (q-1)}{2}) \rho_o^\beta
\leq c r_i^{\gamma_1},
\end{align}
for a constant $c = c(n,q,C_o,C_1,\beta,[\psi]_{0;\beta,\frac{\beta}{2}}, \|u\|_\infty, R)$. We conclude this section by showing that the last estimate holds for an arbitrary radius $r \in (0,R]$. First, let us consider $r \in (0,r_o)$. Choose $i\in \N_0$ such that $r_{i+1} < r < r_i$. By~\eqref{e.osc_u_ri} we find that
\begin{align*}
\essosc_{Q_{r}} u &\leq \essosc_{Q_{r_i}} u \leq c r_i^{\gamma_1} \leq c \left( \tfrac{\delta^i}{4} \boldsymbol{\omega}_o \right)^{\gamma_1 \frac{(q-1)_+}{2}} \hat \rho_i^{\gamma_1} \leq c \eta^{-\gamma_1} \left( \tfrac{\delta^i}{4} \boldsymbol{\omega}_o \right)^{\gamma_1 \frac{(q-1)_+}{2}} \hat \rho_{i+1}^{\gamma_1} \\
& \leq c \left( \delta^\frac{(q-1)_+}{2} \eta \right)^{-\gamma_1} r_{i+1}^{\gamma_1} \leq c r^{\gamma_1}.
\end{align*}
Next, let us assume that $r \in [r_o, \rho_o)$. By~\eqref{e.rho_o}$_2$ and since $\rho_o \leq c(q,n,\|u\|_\infty)r_o$ when $0<q<1$ and $r_o \geq c(q) \rho_o^{1+ \beta\frac{q-1}{2}}$ when $q > 1$ we obtain that
$$
\essosc_{Q_r} u \leq \essosc_{Q_{\rho_o,\rho_o^{2-\eps}}} u \leq 2 \Psi(\rho_o) \leq c \rho_o^{\gamma_o} \leq c \frac{\rho_o^{\gamma_o}}{r_o^{\gamma_1}} r^{\gamma_1} \leq c r^{\gamma_1}.
$$
In the remaining case $r\in [\rho_o, R]$, we have that
$$
\essosc_{Q_r} u \leq \essosc_{Q_{r,r^{2-\eps}}} u \leq 2 \Psi(r) \leq c r^{\gamma_o} \leq c r^{\gamma_1}
$$
Altogether, we conclude that
$$
	\essosc_{Q_{r,r^2}} u
	\leq
	c r^{\gamma_1}
	\quad \text{ for all } r \in (0, R),
$$
which shows that $u$ is H\"older continuous at the arbitrary point $(x_o, t_o) \in \Omega_T$.

\appendix
\section{Rescaling argument} \label{appendix_a}
Let $M > 1$, and consider
$$
	\tilde{u}(x,t) :=
	\frac{1}{M} u(x, M^{q-1} t)
	\quad \text{ and } \quad
	\widetilde{\psi}(x,t) :=
	\frac{1}{M} \psi(x, M^{q-1} t)
$$
for $(x,t) \in \Omega_{\widetilde{T}} := \Omega \times (0,\widetilde{T}) := \Omega \times (0, M^{1-q} T)$.
We claim that $\tilde{u}$ is a weak solution to the obstacle problem
with obstacle $\widetilde{\psi}$ and
$$
	\partial_t \big( |\tilde{u}|^{q-1} \tilde{u} \big)
	- \Div \widetilde{\mathbf{A}}(x,t,\tilde{u}, \nabla\tilde{u}) = 0
	\quad \text{in } \Omega_{\widetilde{T}}
$$
in the sense of Definition~\ref{d.obstacle_weaksol}, where the vector-field
$$
	\widetilde{\mathbf{A}}(x,t,v,\zeta)
	:=
	\frac{1}{M} \mathbf{A}(x, M^{q-1}t, Mv, M\zeta)
$$
satisfies the same structure conditions as $\mathbf{A}$.
For $\widetilde{\psi} \in C^0(\Omega_{\widetilde{T}})$ we compute that
\begin{align*}
	[\widetilde{\psi}]_{C^{0;\beta,\frac{\beta}{2}}}
	&=
	\sup_{(x,\tilde{t}), (y,\tilde{s}) \in \Omega_{\widetilde{T}}}
	\frac{|\widetilde{\psi}(x,\tilde{t}) - \widetilde{\psi}(y,\tilde{s})|}{\max\{ |x-y|^\beta, |\tilde{t} - \tilde{s}|^\frac{\beta}{2} \}} \\
	&=
	\frac{1}{M}
	\sup_{(x,t), (y,s) \in \Omega_T}
	\frac{|\psi(x,t) - \psi(y,s)|}{\max\{ |x-y|^\beta, M^{(1-q)\frac{\beta}{2}} |t-s|^\frac{\beta}{2} \}} \\
	&\leq
	M^{(q-1)_+\frac{\beta}{2} - 1}
	[\psi]_{C^{0;\beta,\frac{\beta}{2}}}.
\end{align*}
Hence, we have that $\widetilde{\psi} \in C^{0;\beta,\frac{\beta}{2}}(\Omega_{\widetilde{T}})$.
Further, it is clear that there holds $\tilde{u} \geq \widetilde{\psi}$ a.e.~in $\Omega_{\widetilde{T}}$ and we compute that
$$
	\tilde{u} \in
	C^0((0,\widetilde{T}); L^{q+1}_{\loc}(\Omega))
	\cap
	L^2_{\loc}(0,\widetilde{T}; H^1_{\loc}(\Omega)),
$$
i.e., we find that $\tilde{u} \in K_{\widetilde{\psi}}(\Omega_{\widetilde{T}})$.
Now, we consider $\widetilde{\varphi} \in C^\infty_0(\Omega_{\widetilde{T}};\R_{\geq 0})$
and $\tilde{v} \in K_{\widetilde{\psi}}'(\Omega_{\widetilde{T}})$.
First, observe that $\varphi(x,t) := \widetilde{\varphi}(x, M^{1-q} t) \in C^\infty_0(\Omega_T;\R_{\geq 0})$.
Furthermore, we state that $v(x,t) := M \tilde{v}(x, M^{1-q} t)$
is an admissible comparison map related to $u$ and $\psi$.
To this end, check that $v \in C^0((0,T); L^{q+1}_{\loc}(\Omega)) \cap L^2_{\loc}(0,T; H^1_{\loc}(\Omega))$ and
$$
	v(x,t) =
	M \tilde{v}(x, M^{1-q} t)
	\geq
	M \widetilde{\psi}(x, M^{1-q} t)
	=
	\psi(x,t)
$$
for a.e.~$(x,t) \in \Omega_T$.
Moreover, compute that
$$
	\partial_t v(x,t)
	=
	\partial_t (M \tilde{v}(x, M^{1-q} t))
	=
	M^{2-q} \partial_t \tilde{v}(x, M^{1-q} t)
	\in L^{q+1}(\Omega_T).
$$
Altogether, this implies that $v \in K_\psi'(\Omega_T)$.
With these considerations at hand, a straightforward computation shows that
\begin{align*}
	\llangle &\partial_{\tilde t} \boldsymbol{\tilde{u}}^q, \widetilde{\varphi} (\tilde{v} - \tilde{u})\rrangle
	+ \iint_{\Omega_{\widetilde{T}}}
	\widetilde{\mathbf{A}}(x,\tilde{t},\tilde{u}, \nabla\tilde{u}) \cdot \nabla \left(\widetilde{\varphi} (\tilde{v} - \tilde{u}) \right) \, \d x \d \tilde{t} \\
	&=
	M^{-(q+1)}
	\left[
	\llangle \partial_t \u^q, \varphi (v - u)\rrangle
	+ \iint_{\Omega_T} \mathbf{A}(x,t,u,\nabla u) \cdot \nabla \left(\varphi (v - u) \right) \, \d x \d t
	\right]
	\geq 0,
\end{align*}
since $u$ is a weak solution associated with the obstacle $\psi$ in the sense of Definition~\ref{d.obstacle_weaksol}.
If we know that $\tilde{u}$ is H\"older continuous with H\"older exponent $\gamma_1 \in (0,1)$, by using the definition of $\tilde{u}$ we compute that also $u$ is H\"older continuous with the same H\"older exponent.

\end{document}